\documentclass[10pt]{article}

\usepackage{amsmath}
\usepackage{amsthm}
\usepackage{amssymb}
\usepackage{graphicx}
\usepackage{hyperref}

\newtheorem{thm}{Theorem}

\newtheorem{remark}{Remark}

\newtheorem{lemma}[thm]{Lemma}
\newtheorem{prop}[thm]{Proposition}

\newcommand{\R}{{\mathbb R}}

\newcommand{\N}{\mathcal{N}}

\newcommand{\nc}{\newcommand}

\nc{\e}{\epsilon}
\nc{\al}{\alpha}
\nc{\be}{\beta}
\nc{\del}{\delta}
\nc{\G}{\Gamma}
\nc{\g}{\gamma}
\nc{\ka}{\kappa}
\nc{\lam}{\lambda}
\nc{\Lam}{\Lambda}
\nc{\Om}{\Omega}
\nc{\Omt}{\tilde{\Omega}}
\nc{\ta}{\tau}
\nc{\w}{\omega}
\nc{\io}{\iota}
\nc{\h}{\theta}
\nc{\z}{\zeta}
\nc{\s}{\sigma}
\nc{\Si}{\Sigma}

\newcommand{\p}{\partial}
\newcommand{\divv}{{\rm div\,}}

\newcommand{\Hess}{{\rm Hess\,}}
\newcommand{\re}{\operatorname{Re}}

\newcommand{\Span}{{\rm span\,}}

\newcommand{\ip}[2]{\left\langle #1, #2 \right\rangle}

\newcommand{\ra}{\rightarrow}
\newcommand{\1}{\mathbf{1}}

\nc{\ran}{\rangle}
\nc{\lan}{\langle}

\newcommand{\DETAILS}[1]{}

\newcommand{\DATUM}{April, 2012}              
\date{April, 2012} 

\title{Stability of spherical collapse under mean curvature flow}

\author{Wenbin Kong$^\dagger$\thanks{This work is a part of this
author's Ph.D. thesis.} and
I.M. Sigal$^\dagger$\thanks{Supported by NSERC Grant NA7901}\thanks{The corresponding author, im.sigal@utoronto.ca}
\\
Department~of Mathematics, University of Toronto, Toronto, Canada }

\begin{document}

\maketitle

\begin{abstract} 
We study the mean curvature flow of hypersurfaces in $\R^{n+1}$, with initial surfaces sufficiently close to the standard $n$-dimensional sphere. The closeness is in the Sobolev norm with the index greater than $\frac{n}{2}+1$ and therefore it does not impose restrictions of the mean curvature of the initial surface. We show that 
the solution of such a flow collapses to a point, $z_*$, in a
finite time, $t_*$, approaching exponentially fast the spheres of radii $\sqrt{2n(t_*-t)}$, centered at $z(t)$, with the latter converging to $z_*$. 

\end{abstract}
\section{Introduction}\label{intro}

We study the behavior of mean curvature flow (MCF) of hypersurfaces in $\R^{n+1}$ with initial conditions close to spheres.
Given an initial simple, closed hypersurface $M_0$ in $\R^{n+1}$, the MCF determines a family $\{M_t |\ t\geq 0\}$ of closed hypersurfaces in $\R^{n+1}$, given by immersions $X(\cdot,t):\Omega\rightarrow \mathbb{R}^{n+1}$, satisfying the
following evolution equation:
\begin{equation}\label{MCF}
\frac{\p X}{\p t}=-H(X)\nu(X),
\end{equation}
where $\Omega\subset\mathbb{R}^{n+1}$ is a  fixed
$n-$dimensional hypersurface,  $\nu(X)$ and
$H(X)$ are the outward unit normal vector and mean curvature at
$X\in M_t$, respectively. We show that if $M_0$ is close to an Euclidean $n$-sphere in the norm $H^s$, $s>\frac{n}{2}+1$,
then the solution $M_t$ collapses to a round point in a finite time. Due to the translation and dilation symmetry of \eqref{MCF},
it suffices to consider $M_0$ close to the standard $n$-sphere (the Euclidean sphere with radius $1$ and center at the origin).

The mean curvature flow 
is the steepest descent flow for the area functional. It arises
in applications, such as models of annealing metals \cite{Mullins} and other problems involving phase separation and moving interfaces (\cite{ESS, Ilmanen1, BK}). It has been recently successfully applied by Huisken and Sinestrari to topological classification of surfaces and submanifolds (\cite{ HS3}, see also \cite{lauer}).  It is closely related to the Ricci and inverse mean curvature flow.

Mean curvature flow was first studied by Brakke \cite{brakke}. Evans and Spruck \cite{ES1}
constructed a unique weak solution of the nonlinear PDE for certain functions whose zero level set evolves in time according to its mean curvature. Similar results were obtained by Chen, Giga and Goto
\cite{CGG} and by Ambrosio and Soner \cite{AS}. The short-time existence in H\"older spaces was proven in \cite{brakke, Huisken, EH2, ES1, Ilmanen2}.
Sharp results on local regularity were established in \cite{W4}. Higher codimension mean curvature flows were studied by Mu-Tao Wang \cite{wangmutao} 
(see also \cite{leelee}). For more results on the existence, uniqueness and regularity of the solution one can see \cite{brakke, ES2, ES3, ilmanen, clutterbuck}.

The question of the long time existence is, as usual, more subtle. Ecker and Huisken \cite{EH1} showed longtime existence for mean curvature flow in the case of linearly growing graphs. Later they \cite{EH2} proved that if the initial surface is a locally Lipschitz continuous entire graph over $\R^n$ then the solution will exist for all times.

However, the most interesting aspect of the mean curvature flow is formation of singularities, with two canonical examples being 
Eucledian spheres or cylinders,  collapsing to their center or axis, respectively, 
 their radii evolving as $\sqrt{2n(t_*-t)}$ or $\sqrt{2(n-1)(t_*-t)}$. 
 These two cases suggest collapse to a round point or round line as two possible scenarios of formation of singularities.  
 The former scenario, indeed, showed up in many works, starting with the result,
 due to Gage and Hamilton
\cite{GH}, who showed that  initial convex plane curves shrink to a 'round' point, i.e. approach asymptotically circles of radii $\sqrt{2(t_*-t)}$. Later, Grayson \cite{grayson} showed that any embedded plane curves always shrink smoothly until they are convex, and
then to points by the evolution theorem of convex curves. For higher dimensions the latter result does not hold. In a seminal work, \cite{Huisken}, Huisken showed that under mean curvature flow a convex hypersurface in $\R^n$, $n\geq 3$, shrinks smoothly to a point, getting spherical in the limit. These result was extended in \cite{Huis2, HS1, HS2, W1, W3}.

It was conjectured by Huisken that starting at a generic smooth closed embedded surface in $\R^3$, the mean curvature flow remains smooth until it arrives at a singularity in a neighborhood of which the flow looks like concentric spheres or cylinders. 
A part of this conjecture was proved in \cite{CM}, where it was shown that the only singularities which cannot be perturbed away are spheres and cylinders. For more results see \cite{ cooper, LS1, LS2}. The present work shows that the collapsing sphere solutions are stable:

\begin{thm}\label{mainthm}
Let $\Omega=S^n$ be  the standard $n$-dimensional sphere and let a surface $M_0$, defined by an immersion $x_0\in H^s(\Omega) $, for some
$s>\frac{n}{2}+1$, be close to $S^n$, in the sense that $\|x_0-\1\|_{H^s}\ll 1$.
\DETAILS{in the Sobolev metric with the index
greater than $\frac{n}{2}+1$ (and be symmetric with respect to the hyperplanes
$x_i=0,\ i=1, \cdots, n+1$)}
Then there exist $t_*<\infty$ and $z_*\in \R^{n+1}$, s.t. \eqref{MCF} has the unique solution, $M_t,\ t<t_*$, and this solution contracts to the point $z_*$, as $t_* \ra \infty$. Moreover, $M_t$ is defined by an immersion $x(\cdot, t)\in H^s(S^n) $, with the same $s$, of the form
$$x(\omega, t)= z(t)+R(\omega, t)\omega,$$
for some $z(t)\in \R^{n+1}$ and $R(\cdot, t)\in H^s(S^n)$, satisfying
$z(t)=z_*+O((t_*-t)^{\frac{1}{2a_*}(n+\frac{1}{2}-\frac{1}{2n})})$ and
\begin{equation} \label{St}
R(\omega, t)=\lambda(t)(\sqrt{\frac{n}{a(t)}}+\xi(\omega, t)),
\end{equation} with
$\lambda(t), a(t)$ and  $\xi(\cdot, t)$ which satisfy
$$\lambda(t)=\sqrt{2a_*(t_*-t)}+O((t_*-t)^{\frac{1}{2}+\frac{1}{2a_*}(1-\frac{1}{2n})}),$$
$a(t)=-\lambda(t)\dot{\lambda}(t)=a_*+O((t_*-t)^{\frac{1}{2a_*}(1-\frac{1}{2n})})$ 
and $\|\xi(\cdot, t)\|_{H^s}\lesssim (t_*-t)^{\frac{1}{2n}}$.
Moreover, $|z_*| \ll 1$.
\end{thm}
Note that our condition on the initial surface does not impose any  restrictions of the mean curvature of this surface

In contrast to the above result, it was shown in \cite{gangsigal}, that for  an open set of  initial conditions arbitrary close to an infinite cylinder, the mean curvature flow does not converge to a (round) line, but develops a singularity at a point in a finite time  ('neck-pinching'), provided initial conditions have an arbitrary shallow neck. 
Thus, unlike spheres, the cylinders are not stable under the mean curvature flow. 
%
(For earlier results dealing with the neck-pinching for compact ( barbell shaped) or periodic (torus-like) surfaces see~\cite{AlAnGi1, AlAnGi, av, af, ath, ecker, ES1, CGG, dk,  Huis2, SS, SM} 
and references therein.)

The form of expression \eqref{St} above is a reflection of a large class of symmetries of the mean curvature flow:
\begin{itemize}
\item \eqref{MCF} is invariant under rigid motions of the surface,
i.e. $X\mapsto RX+a$, where $R\in O(n+1)$, $a\in \R^{n+1}$ and
$X=X(u,t)$ is a parametrization of $S_t$, is a symmetry of
\eqref{MCF}.
\item \eqref{MCF} is invariant under the scaling $X\mapsto \lambda
X$ and $t\mapsto \lambda^{-2} t$ for any $\lambda>0$.
\end{itemize}
Our approach utilizes these symmetries in an essential way. It uses the rescaling of the equation \eqref{MCF} by a parameter $\lambda(t)$ whose behaviour is determined by the equation itself and a series of differential inequalities for a Lyapunov-type functions.

\begin{remark} If the initial condition $x_0$ is invariant under the transformation $x_i\rightarrow -x_i$
for any $i=1, \cdots, n+1$, then $z(t)=0$ and the proof below simplifies considerately.
\end{remark}

This paper is organized as follows. We
rescale the equation \eqref{MCF} in Section \ref{sec:rescSurf} by introducing collapse variables,
designed so that the new equation has global solutions 
and reformulate the main theorem in terms of the rescaled surfaces. 
In the same section we present the equation for the surface as a normal graph over a sphere. This equation is derived in Appendix A. In Section \ref{sec:centerz} we introduce a notion of the 'center' of a surface,
close to a unit sphere in $\R^{n+1}$ and show that such a center exists.
We will show in Section \ref{sec:mainproof} that the centers $z(t)$ of the solutions to \eqref{MCF} converge
to the collapse point, $z_*$, of Theorem \ref{mainthm}. 
In Section \ref{sec:reparametrization} we reparametrize the solutions of the new equation by
isolating the leading term and a perturbation. In Section
\ref{sec:decomposition} we use the Lyapunov-Schmidt type decomposition to derive equations for the parameters and perturbation. In Section
\ref{sec:linearopr} we discuss the spectrum of the linearized equation. In Section \ref{sec:functional} we introduce certain Lyapunov functionals and derive differential
inequalities for them. These inequalities are used in Section
\ref{sec:mainproof} to derive a priori bounds on Sobolev norms of the perturbation. In Section \ref{sec:mainproof} we prove the main theorem.

\textbf{Notation.} The relation $f \lesssim g$ for positive functions $f$ and $g$ signifies that there is a numerical constant $C$, s.t. $f \le Cg$.

\section*{Acknowledgements}
The second author is grateful to D. Knopf for useful discussions. A part of this work
was done while the second author was at IAS and while visiting ETH Z\"urich and
Edwin Schr\"odinger Institute, Vienna.

\section{Rescaled equation}\label{sec:rescSurf}
Instead of the surface $M_t$, it is convenient to consider the new, rescaled surface $\tilde M_\tau =\lambda^{-1}(t)(M_t-z(t))$, where $\lambda(t)$ and $z(t)$ are some differentiable functions to be determined later, and $ \tau=\int_0^t \lambda^{-2}(s) ds$. The new surface is described by $y$, which is, say, an immersion of some fixed
$n-$dimensional hypersurface $\Omega\subset\mathbb{R}^{n+1}$, i.e.
$y(\cdot, \tau):\Omega\rightarrow \mathbb{R}^{n+1}$, (or a local parametrization of $\tilde M_\tau$, i.e. $y(\cdot, \tau):U\rightarrow M_t$). Thus the new
collapse variables are given by
\begin{equation} \label{collapsevar}
y(\omega,\tau)=\lambda^{-1}(t)(X(\omega, t)-z(t))
\ \text{and} \ \tau=\int_0^t \lambda^{-2}(s) ds.
\end{equation}
Let $\dot{\lambda}=\frac{\p \lambda}{\p t}$ and $\frac{\p z}{\p \tau}$ be the $\tau$-derivative of $z(t(\tau))$, where $t(\tau)$ is the inverse function of $\tau(t)=\int_0^t \lambda^{-2}(s) ds$. Using that $\frac{\p X}{\p t}=\frac{\p z}{\p t}+
\dot{\lambda}y+\lambda\frac{\p y}{\p \tau}\frac{\p \tau}{\p t}=\lambda^{-2}\frac{\p z}{\p \tau}+\dot{\lambda}y+\lambda^{-1}\frac{\p y}{\p \tau}$
and $H(\lambda y)=\lambda^{-1}H(y)$, we obtain from
\eqref{MCF} the equation for $y$, $\lambda$ and $z$:
\begin{equation}\label{eqn:y}
\frac{\p y}{\p \tau}=-H(y)\nu(y)+ay-\lambda^{-1}\frac{\p z}{\p \tau}\ \text{and}\ a=-\lambda\dot{\lambda}.
\end{equation}

In what follows we take $\Omega$ to be $S^n$, the unit sphere centered at the origin.
In this case, the equation  \eqref{eqn:y} has static solutions ($a=$  a positive constant, $z=0$, $y(\omega)=\sqrt{\frac{n}{a}}\omega$).

Standard results on the local well-posedness for the mean curvature flow (see e.g. \cite{Lieberman}, Theorem 8.3, and also \cite{Eidelman, Friedman}) imply that for an initial condition $y_0 \in C^\alpha,\ \alpha >1,$ and given functions $a(\tau),\ z(\tau) \in C^1 \cap L^\infty (\R)$, there is $T>0$, s.t. \eqref{eqn:y} has a unique solution, $y \in C^\alpha$, on the time interval $[0, T)$ and either $T=\infty$ or $T< \infty$ and $\|y\|_{C^\alpha} \ra \infty$ and $\tau \ra T$. Here $C^\alpha$ is the space of $[\alpha] (=$ the integer part of $\alpha$) times differentiable functions on $S^n$, 
whose highest derivatives are H\"older continuous with the index $\alpha-[\alpha]$. This result extends also to $ H^s(S^n) $ with $s>\frac{n}{2}+1$.

\DETAILS{\section{Graph representation of $S_t$}\label{sec:graph}

It is convenient to reparametrize the surface $S_t$ as a graph over $n+2$ functions,  $\rho(\cdot, t): \Omega \ra \R^+$ and $z \in \R^{n+1}$, as follows.
Let $S^n$ be the unit sphere centered at the origin. We say that a hypersurface $S$ is a $(\rho, z)-$graph in normal direction over $\Omega$, and  denote
$S=graph(\rho,z)$, if it can written in the form  
%
\begin{equation}\label{Srhograph}
S=\{z+ \rho(\omega)\omega\ |\omega \in \Omega\}.
\end{equation}
(In this case, $x(\omega)=z+ \rho(\omega)\omega$ is an immersion of $\Omega$.)}

Our goal is to prove the following result.
\begin{thm}\label{thmrho}
Let $\rho_0\in H^s(S^n) $
satisfy $\|\rho_0-\1\|_{H^s}\ll 1$ for some
$s>\frac{n}{2}+1$, and let $\lambda_0>0$ and $|z_0|\ll 1$. Then \eqref{eqn:y} with initial data
$(y_0=\rho_0(\omega)\omega, \lambda_0, z_0)$ has a unique solution $(y, \lambda, z)$ for $\forall \tau$, with $y(\omega,\tau)$ 
of the form $y(\omega,\tau)=\rho(\omega,\tau)\omega$, with $\rho(\cdot, \tau)\in H^s(S^n),\ \rho(\omega,\tau)=\sqrt{\frac{n}{a(\tau)}}+\xi(\omega,\tau)$, and $\lambda(\tau)$ and $z(\tau)$ satisfying
$\lambda(\tau)=\lambda_0 e^{-\int_0^\tau a(s)ds}$ and $|z(\tau)-z_*| \lesssim e^{-(n+\frac{1}{2}-\frac{1}{2n})\tau}$,
for some $a(\tau)=a_*+O(e^{-(1-\frac{1}{2n})\tau})$, with $|a_*-n|\leq \frac{1}{2}$ and $|z_*|\ll 1$.
\end{thm}
This Theorem together with \eqref{collapsevar} implies Theorem \ref{mainthm} (see Section \ref{sec:mainproof}).

In what follows  $g_{ij}$ is the standard metric on $S^n$  (induced by the inner product in $\R^{n+1}$) 
and $\Delta$ 
is the Laplace-Beltrami operator in this metric. 
Furthermore we define $(\Hess\rho)_{ij}=\frac{\p^2\rho}{\p u^iu^j}-\Gamma_{ij}^k\frac{\p \rho}{\p u^k}$,
$\Gamma_{ij}^k=\frac{1}{2}g^{kn}(\frac{\p g_{in}}{\p u^j}+\frac{\p g_{jn}}{\p u^i}-
\frac{\p g_{ij}}{\p u^n})$, and 
$\nabla^k\rho= g^{km} \frac{\p \rho}{\p u^m}$ in a local
parametrization $x=x(u)$ of $S^n$ (so that  $g_{ij}:=\frac{\p x^{k}}{\p
u^i}\frac{\p x^{k}}{\p u^j}$). Here and in what follows the summation over the repeated
indices is assumed. (Note that $(Hess)_{ij}=\nabla_i\nabla_j$, where $\nabla_i\rho=\frac{\p \rho}{\p u^i}$ and $(\nabla_i\omega)_j=\frac{\p\omega_j}{\p u^i}-\Gamma_{ij}^k\omega_{k}$.) In the appendix we prove the following

\begin{prop}\label{prop:eqnrho}
Let $\tilde M_\tau =\lambda^{-1}(t)(M_t-z(t))$ 
be defined by an immersion $y(\omega, \tau)=\rho(\omega, \tau)\omega$ of $S^n$ for some 
functions $\rho(\cdot, \tau): S^n\rightarrow\R^+$, differentiable in their arguments and let $z(\tau)\in C^1(\R^+, \R^{n+1})$. Then
$\tilde M_\tau$ satisfies \eqref{eqn:y} if and only if $\rho$ and $z$ satisfy the
equation
\begin{equation}\label{eqnrho}
\frac{\p \rho}{\p \tau}=G(\rho)+a\rho-\lambda^{-1}z_{\tau}\cdot\omega+\lambda^{-1}\tilde z_\tau\cdot\frac{\nabla\rho}{\rho},
\end{equation}
where ${\tilde{z}}_{\tau k}=\frac{\p x^i}{\p u^k}z_\tau^i,\ z_\tau:= \frac{\p z}{\p \tau}$ and
\begin{equation}\label{G}
G(\rho)=\frac{1}{\rho^2}\Delta
\rho-\frac{n}{\rho}-\frac{\nabla\rho \cdot \Hess(\rho)\nabla\rho}{\rho^2(\rho^2+|\nabla\rho|^2)}
+\frac{|\nabla\rho|^2}{\rho(\rho^2+|\nabla\rho|^2)}.
\end{equation}
\end{prop}
\section{Collapse center} 
\label{sec:centerz}

In this section we introduce a notion of the 'center' of a surface, close to a unit sphere, $S^n$, in $\R^{n+1}$, and show that such a center exists. We will show in Section \ref{sec:mainproof} that the centers $z(t)$ of the solutions $M_t$ to \eqref{MCF} converge to the collapse point, $z_*$, of Theorem \ref{mainthm}. For a closed surface $S$, given by an immersion $x: S^n\ra\R^{n+1}$, we define the center, $z$, by the relations $\int_{S^n} ((x-z)\cdot \omega) \omega^j=0,\  j=1, \dots, n+1$. The reason for this definition will become clear in Section \ref{sec:linearopr}. We have

\begin{prop} \label{prop:centerz}
Assume a surface $M$ is given by an immersion $x: S^n \ra \R^{n+1}$, with $y:=\lambda^{-1}(x-\bar z)\in H^1(S^n, \R^{n+1})$ close, in the $H^1(S^n, \R^{n+1})$-norm,  to the identity $\1$, for some $\lambda \in \R^+$ and $\bar z\in \R^{n+1}$. 
Then there
exists  $z\in \R^{n+1}$ such that $\int_{S^n} ((x-z)\cdot \omega) \omega^j=0,\  j=1, \dots, n+1$.
\end{prop}

\begin{proof}  By replacing $x$ by $x^{new}$, if necessary, we may assume that $\bar z=0$ and $\lambda=1$.
Let $x \in H^1(S^n, \R^{n+1})$.
The relations  $\int_{S^n} ((x-z)\cdot \omega) \omega^j=0 \ \forall j$
are equivalent to the equation $F(x, z)=0$, where $F(x, z)=(F_1(x, z), \dots, F_{n+1}(x, z))$, with
$$F_j(x, z) =\int_{S^n} ((x-z)\cdot \omega) \omega^j,\ j=1, \dots, n+1.$$
Clearly  $F$ is a $C^1$ map
from $H^1(S^n, \R^{n+1})\times \R^{n+1}$ to $\R^{n+1}$.
We notice that $F(\1, 0)=0$.  We solve the equation $F(x, z)=0$ near $(\1, 0)$, using the implicit function theorem.
To this end we calculate the derivatives 
$\p_{z^i}F_j=-\int_{S^n}\omega^i\omega^j
=-\frac{1}{n+1}\delta_{ij}|S^n|$ for $j=1, \cdots, n+1$.
The above relations allow us to apply implicit function theorem to show that for any $x$ close to $\1$, there exists $z$, close to $0$, such that $F(x, z)=0$.
 \end{proof}

 Assume we have a family, $x(\cdot, t): S^n \ra \R^{n+1},\ t\in [0, T]$, of immersions and functions  $\bar z(t)\in \R^{n+1}$ and $\lambda (t) \in \R^+$, s.t. $\lambda^{-1}( t)(x(\omega, t)- \bar z(t))$, in the $H^1(S^n, \R^{n+1})$-norm,  to the identity $\1$ (i.e. a unit sphere).  Then Proposition \ref{prop:centerz} implies that there exists $z(t) \in \R^{n+1}$, s.t.
 \begin{equation}\label{orthog'}
\int_{S^n} ((x(\omega, t)- z( t))\cdot \omega) \omega^j=0,\  j=1, \dots, n+1.
\end{equation}
 Furthermore, 
if $y(\omega, \tau):=\lambda^{-1}( t)(x(\omega, t)- z( t))= \rho(\omega, \tau)\omega$, where $\tau=\tau (t)$ is given in \eqref{collapsevar}, then we conclude that
\begin{equation} \label{rhoomegajorthog}
\int_{S^n} \rho(\omega, \tau) \omega^j=0,\  j=1, \dots, n+1.
\end{equation}

To apply the above result to the 
immersion $x(\cdot, t): S^n \ra \R^{n+1}$, solving \eqref{MCF}, we pick $\bar z (t)$ to be a piecewise constant function
constructed iteratively, starting with $\bar z (t)=0$ for $0\le t \le \delta$ for $\delta$ sufficiently small (this works due to our assumption on the initial conditions), and $\bar z (t)=z(\delta)$ for $\delta \le t \le \delta+\delta'$ and so forth (see Section \ref{sec:mainproof}). This gives $z(t) \in \R^{n+1}$, s.t. \eqref{orthog'} holds.
This is $z(t)$ we use in \eqref{collapsevar}.

\section{Reparametrization of solutions}\label{sec:reparametrization}

The next proposition will be used to reparamtrize the initial condition for \eqref{eqnrho}.
\begin{prop}\label{prop:decomposition}
If $\|\rho-\sqrt{\frac{n}{a_\#}}\|_{L^1}\leq \delta :=\frac{(n-\frac{1}{2})^{5/2}\sqrt{n}|S^n|}{6a_\#^{3}}$ for some
$n-\frac{1}{2}<a_\#<n+\frac{1}{2}$, then there exists $a=a(\rho)$ s.t.
\begin{equation} \label{vorthof1}
\rho-\sqrt{\frac{n}{a}}\perp 1\ \text{in}\  L^2(S^n).
\end{equation}
Moreover, $|a(\rho)-a_\#|\lesssim \|\rho-\sqrt{\frac{n}{a_\#}}\|_{L^1}$ and
$\|\rho-\sqrt{\frac{n}{a}}\|_{H^s}\lesssim \|\rho-\sqrt{\frac{n}{a_\#}}\|_{H^s}$.
\end{prop}

\begin{proof}
The orthogonality conditions on the fluctuation can be written as
$F(\rho,a)=0$, where $F:
L^2(\Omega)\times\R^{+}\rightarrow\R$ is defined as
$F(\rho,a)=\int_{S^n} (\rho-\sqrt{\frac{n}{a}})$. 

Note first that the mapping $F$ is $C^1$ and
$F(\sqrt{\frac{n}{a}}, a)=0\ \forall a$. We compute the linear map $\p_a F(\rho, a)$: 
\begin{equation} \label{F'}
(\p_a F)(\rho, a)=\int_{S^n} \frac{\sqrt{n}}{2}a^{-3/2} 
=\frac{\sqrt{n}}{2a^{3/2}}|\Omega|.
\end{equation}
Hence $\p_a F(\rho, a)$ is
invertible. Thus, by implicit function theorem, the equation $F(\rho,a)=0$ has a unique solution for $a$ in a neighbourhood of the point $(\sqrt{\frac{n}{a_\#}}, a_\#)$.

To obtain estimates on the neighbourhood above we follow through the proof of the implicit function theorem. We expand the function $F(\rho, a)$ in $a$ around $a_\#$:
\begin{equation} \label{Fexp}
F(\rho, a)=F(\rho,a_\#)+\p_a F(\rho,a_\#)(a-a_\#)+R(\rho,a),
\end{equation}
where $R(\rho,a)$ is defined by this equation. Hence, by the standard remainder formula, 
it satisfies, for $|a-a_\#|,\ |a'-a_\#|\le r$,
\begin{equation}\label{eqn:estR}
|R(\rho,a)|\le \frac{1}{2}\sup|\p^2_a F| r^2 \le
\frac{3\sqrt{n}|S^n|}{8(n-1/2)^{5/2}}r^2
\end{equation}
 and
\begin{equation}\label{eqn:differenceR}
\begin{array}{lll}
|R(\rho,a')&-R(\rho,a)|=|F(\rho,a')-F(\rho,a)-\p_a F(\rho, a_\sharp)(a'-a)|\\
&=  |\int_0^1 \p_s F(\rho,sa'+(1-s)a)ds-\p_a F(\rho, a_\sharp)(a'-a)|\\
&= \int_0^1 ds |\p_a F(\rho,sa'+(1-s)a)-\p_a F(\rho, a_\sharp) ||a'-a|\\
 &\le \sup|\p_a^2 F(\rho,a)|r|a'-a|\\
 &\le  \frac{3\sqrt{n}|S^n |}{8(n-1/2)^{5/2}}r|a'-a|.
\end{array}
\end{equation}
Using  \eqref{Fexp}, we rewrite the equation $F(\rho,a)=0$ as a fixed point problem
$a-a_\#=\Phi_\rho(a-a_\#)$, where $\Phi_\rho(a-a_\#)=-\p_a F(\rho,a_\#)^{-1}[F(\rho,a_\#)+R(\rho,a)]$.
Choose $r=(\frac{8(n-\frac{1}{2})^{5/2}}{3\sqrt{n}|S^n|}\delta)^{1/2}
=\frac{2(n-\frac{1}{2})^{5/2}}{3a_\#^{3/2}}$. Then if $|a-a_\#|\le r$, we have by \eqref{eqn:estR}
$$|\Phi_\rho(a-a_\#)| \le \frac{2a_\#^{3/2}}{\sqrt{n}|S^n|}(\delta+\frac{3\sqrt{n}|S^n|}{8(n-\frac{1}{2})^{5/2}}r^2)
=\frac{4a_\#^{3/2}}{\sqrt{n}|S^n|}\delta = r.$$
Moreover, by \eqref{eqn:differenceR}
\begin{equation*}
\begin{array}{lll}
|\Phi_\rho(a'-a_\#)-\Phi_\rho(a-a_\#)|&=|\p_a F(\rho,a_\#)^{-1}||R(\rho,a')-R(\rho,a)|\\
&\le  \frac{2a_\#^{3/2}}{\sqrt{n}|S^n|}\frac{3\sqrt{n}|S^n|}{8(n-1/2)^{5/2}}r|a'-a|
= \frac{1}{2}|a'-a|.
\end{array}
\end{equation*}
we conclude that this fixed point problem has a unique solution satisfying the estimates $|a-a_\#|\le r$, provided $ \|\rho-\sqrt{\frac{n}{a_\#}}\|_{L^1} \le \delta$.
\end{proof}

\DETAILS{
\begin{prop}\label{decomposition}
If $\|v-\frac{1}{\sqrt{a_0}}\|\ll 1$ for some
$\frac{1}{2}<a_0<\frac{3}{2}$ in $\R^2$, then there exists $a=a(v)$ so that
$$v-\frac{1}{\sqrt{a}}\perp 1\ in\  L^2(\Gamma).$$
Similarly, if
$\|v-\sqrt{\frac{2}{a_0}}\|\ll 1$ for some
$1<a_0<3$ in $\R^3$, then there exists $a=a(v)$ so that
$$v-\sqrt{\frac{2}{a}}\perp 1\ in\  L^2(\Gamma).$$
Moreover, $|a(v)-a_0|\ll 1$ in both cases.
\end{prop}

\begin{proof}
We will prove the case in $\R^2$. The case in $\R^3$ can be proved
exactly in the same way.
The orthogonality conditions on the fluctuation can be written as
$F(v,a)=0$, where $F:
L^2(\Gamma)\times\R^{+}\rightarrow\R$ is defined as
$F(v,a)=\ip{v-\frac{1}{\sqrt{a}}}{1}$.
Here and in what follows, all inner products are the $L^2$ inner
products.

Note first that the mapping $F$ is $C^1$ and
$F(\frac{1}{\sqrt{a_0}}, a_0)=0$. We claim that the linear map $\p_a
F(\frac{1}{\sqrt{a_0}}, a_0)$ is invertible. Indeed, we compute
\begin{equation*}
\p_a
F(\frac{1}{\sqrt{a_0}},
a_0)=\ip{\frac{1}{2}a^{-3/2}}{1}|_{a=a_0}=
\pi a_0^{-3/2}
\end{equation*}
Hence $\p_a F(\frac{1}{\sqrt{a_0}}, a_0)$ is
invertible, and by implicit function theorem we prove the first part
of the proposition.

Next we expand the function $F(v, a)$ in $a$ around $a_0$:
$$F(v, a)=F(v,a_0)+\p_a F(v,a_0)(a-a_0)+R(v,a),$$
where $R(v,a)=O(|a-a_0|^2)$ uniformly in $v$. By this we have
$|a-a_0|\lesssim |F(v,a_0)|+|a-a_0|^2$. From the condition
$\|v-\frac{1}{\sqrt{a_0}}\|<<1$ it follows that $|a-a_0|<<1$.
\end{proof}}

Unfortunately, we cannot apply Proposition \ref{prop:decomposition} directly to solutions, $\rho(\omega, \tau)$,
of the equation \eqref{eqnrho}, since the parameter-function $a(\tau)$ which would come out of
Proposition \ref{prop:decomposition} would have to be equal to the $a(\tau)$ entering \eqref{eqnrho}.
To overcome this problem, we 'deconstruct' $\rho(\omega, \tau)$, using that it originates
as $\rho(\omega,\tau):= \lambda(t)^{-1}R(\omega, t)$, with $R(\omega, t):= (x(\omega, t)-z(t))\cdot \omega$
(see Equation \eqref{collapsevar} and Theorem \ref{thmrho})
and $\tau=\tau(t)$ given by \eqref{collapsevar},
and we construct an orthogonal decomposition for $R(\omega, t)$. 

For any time $t_0$ and constant $\delta>0$we define $I_{t_0,\delta}:=[t_0, t_0+\delta]$ and
$$\mathcal{A}_{t_0,\delta}:=C^1(I_{t_0,\delta}, [n-\frac{1}{2}, n+\frac{1}{2}]).$$
Let $\lambda(t)$ be positive, differentiable function and denote $R_\lambda(\omega,t):= \lambda(t)^{-1}R(\omega, t)$. 
For any function $a\in \mathcal{A}_{t_0,\delta}$ and $\lambda_0 >0$, we define the positive function
$$\lambda(a,\lambda_0)(t):=(\lambda_0^2-2\int_{t_0}^t a(s)ds)^{1/2},$$
i.e.  $\lambda(a,\lambda_0)(t)$ satisfies $\lambda(t)\p_t \lambda(t)=a(t)$ and  $\lambda(t_0)=\lambda_0$.
Suppose $R$ is
such that 
\begin{equation}\label{Udefn}
\sup_{t\in I_{t_0,\delta}} \|R_{\lambda(\bar a, \lambda_0)}(t)-\sqrt{\frac{n}{\bar a(t)}}\| \ll 1,
\end{equation}
for some $a\in\mathcal{A}_{t_0,\delta}$ and $ \lambda_0>0$. We define the set
$$\mathcal{U}_{t_0,\delta,\lambda_0}:=\{ R \in C^1(I_{t_0,\delta}, L^2(\Omega))\ |\
\eqref{Udefn}\ \mbox{holds for some}\ a(t)\}.$$
In what follows, all inner products are the $L^2$ inner products.
\begin{prop}\label{prop:timedecom}
Suppose  $\lambda_0^{-2}\delta \ll 1$. Then there exists a unique $C^1$
map $g: \mathcal{U}_{t_0,\delta,\lambda_0} \rightarrow \mathcal{A}_{t_0,\delta}$, such that for $t\in I_{t_0,\delta}$, any  $R\in \mathcal{U}_{t_0,\delta,\lambda_0}$ can be uniquely represented in the form
\begin{equation}\label{vlambda}
R_{\lambda}(\omega,t)=\sqrt{\frac{n}{a(t)}}+\xi(\omega, t),
\end{equation}
with $g(R)(t)=a(t)$, $\xi(\cdot, t)\perp 1$ in $L^2(\Omega)$, and $\lambda(t)=\lambda(a, \lambda_0)$. 
\end{prop}

\begin{proof} In this proof we write $\lambda(a)$ instead of $\lambda(a, \lambda_0)$ and $R_{\lambda(a)}(t)$ for the function $\omega \ra R_{\lambda}(\omega,t)$. Define the $C^1$ map
$G: C^1(I_{t_0,\delta}, [n-\frac{1}{2}, n+\frac{1}{2}])\times C^1(I_{t_0,\delta}, L^2(\Omega))
\rightarrow C^1(I_{t_0,\delta}, \R)$ as
$$G(a,R)(t):=\ip{R_{\lambda(a)}(t)-\sqrt{\frac{n}{a(t)}}}{1}.$$
The orthogonality condition on the fluctuation can be written as $G(a,R)=0$. We solve this equation using the implicit function theorem.
Note first that $G(a,\sqrt{\frac{n}{a}})=0,\ \forall a$. Next, we 
compute
$$\p_a G=\ip{\frac{\sqrt{n}}{2}a^{-3/2}}{1}+\ip{\p_a R_{\lambda(a)}}{1}.$$
Note that $\p_a  R_{\lambda(a)} \alpha=\lambda(t)^{-2}R_{\lambda(a)}\int_{t_0}^t \alpha(s)ds.$
Using this expression and the inequality $\lambda(t)\geq \frac{\lambda_0}{\sqrt{2}}$, provided
that $\delta \leq (4\|a\|_{\infty})^{-1}\lambda_0^2$, we estimate
$$\|\p_a  R_{\lambda(a)} \alpha\|_{\infty} \lesssim \delta \lambda_0^2 \|R_\lambda\|_{\infty}\|\alpha\|_{\infty}.$$
So $\p_a  R_{\lambda(a)}$ is small, if $\delta \ll (\lambda_0^2 \|R_\lambda\|_{\infty})^{-1}$. This shows that $\p_a G(a,R)$
is invertible, provided $R_{\lambda(a)}$ is close to $\sqrt{\frac{n}{a}}$. Hence the implicit function theorem implies that for any $\bar a \in \mathcal{A}_{t_0,\delta}$ there exists
a neighborhood $\mathcal{U}_{\bar a}$ of $\sqrt{\frac{n}{\bar a}}$ in $ C^1(I_{t_0,\delta}, L^2(\Omega))$ and a unique $C^1$ map $g: \mathcal{V}_{\bar a}:=\{R\in  C^1(I_{t_0,\delta}, L^2(\Omega))\ |R_{\lambda(\bar a)}\in \mathcal{U}_{\bar a}\}
\rightarrow \mathcal{A}_{t_0,\delta}$, such that $G(g(R),R)=0$ for all $R\in \mathcal{V}_{\bar a}$.
 Proceeding as in the proof of Proposition \ref{prop:decomposition}, we obtain a quantative description of the neighbourhood , which implies the statement 
 of Proposition \ref{prop:timedecom}.
\end{proof}

Let an immersion $x(\cdot, t): \Omega \ra \R^{n+1}$ satisfy \eqref{MCF}
and let  $z(t) \in \R^{n+1}$ be as in Section \ref{sec:centerz}, i.e. such that \eqref{orthog'} holds.
We apply Proposition \ref{prop:timedecom} to $R(\omega, t)=(x(\omega, t)-z(t))\cdot \omega$  to obtain $a( \tau) $ and $\xi(\omega, \tau)$  s.t. $\rho(\omega, \tau)\equiv R_{\lambda( a, \lambda_0)}(\omega, t)$ satisfies
\begin{equation}\label{rho:orthdec}
\rho(\omega, \tau)=\sqrt{\frac{n}{a(\tau)}}+\xi(\omega, \tau),
\end{equation}
with $\xi \perp 1$. (Here in some functions we changed the time $t$ to $\tau =\tau(t)$, given in \eqref{collapsevar}.)
This together with
\eqref{rhoomegajorthog} 
implies that  $\int_\Omega (\rho-\sqrt{\frac{n}{a}})\omega^j=0,\ j=0, \dots, n+1,$ where we use
the notation $\omega^0=1$, or
\begin{equation}\label{orthv}
\rho-\sqrt{\frac{n}{a}}\perp \omega^j,\ j=0, \dots, n+1,\ \text{in}\  L^2(\Omega).
\end{equation}

\section{Lyapunov-Schmidt decomposition}\label{sec:decomposition}
Let $\rho$ solve  \eqref{eqnrho} and 
assume it can be written as $\rho(\omega, \tau)=\rho_{a(\tau)}+\xi(\omega, \tau)$, with $\rho_a=\sqrt{\frac{n}{a}}$ and $\xi\perp  \omega^j,\ j=0, \dots, n+1$.
Plugging this into equation \eqref{eqnrho}, we obtain the equation
\begin{equation}\label{eqn:xi}
\frac{\p \xi}{\p \tau}=-L_a \xi+N(\xi)+\lambda^{-1}\tilde{z}_\tau\cdot\frac{\nabla \xi}{\rho_a+\xi}+{F},
\end{equation}
where $L_a=-\p G(\rho_a)$, ${N}(\xi)=G(\rho_a+\xi)-G(\rho_a)-\p G(\rho_a)\xi$, ${F}=-\p_\tau \rho_a-\lambda^{-1}z_\tau\cdot\omega$ and, recall, $z_\tau=\frac{\p z}{\p\tau}$ and $\tilde{z}_{\tau k}=\frac{\p x^i}{\p u^k}z_{\tau}^i$. Let $a_\tau=\frac{\p a}{\p \tau}$. We compute
\begin{equation}\label{LNF}
\begin{array}{ll}
& L_a=\frac{a}{n}(-\Delta-2n),\\
& {N}(\xi)=-\frac{(\rho_a+\rho)\xi\Delta\xi}{\rho^2\rho_a^2}-\frac{n\xi^2}{\rho\rho_a^2}+\frac{|\nabla\xi|^2}{\rho(\rho^2+|\nabla\xi|^2)}
-\frac{\nabla\xi\cdot \Hess(\xi)\nabla\xi}{\rho^2(\rho^2+|\nabla\xi|^2)},\\
& {F}=\frac{\sqrt{n}}{2}a^{-3/2}a_\tau
-\lambda^{-1}z_\tau\cdot\omega.
\end{array}
\end{equation}

Now, we project \eqref{eqn:xi}
onto $span\{\omega^j,\ j=0, \dots, n+1\}$. By
\begin{equation}\label{xiorthog}
\xi\bot \omega^j,\ j=0, \dots, n+1,
\end{equation} 
we have
\begin{equation}\label{projection1}
\frac{\sqrt{n}}{2}a^{-3/2}a_\tau|S^n|=\ip{{N}(\xi)+\lambda^{-1}\tilde{z}_\tau\cdot\frac{\nabla \xi}{\rho_a+\xi}}{1},
\end{equation}
\begin{equation}\label{projectionomegaj}
-c\lambda^{-1}z^j_\tau=\ip{{N}(\xi)+\lambda^{-1}\tilde{z}_\tau\cdot\frac{\nabla \xi}{\rho_a+\xi}}{\omega^j},\ j=1, \dots, n+1,
\end{equation}
where $c:=\int_\Omega (\omega^j)^2=\frac{1}{n+1}|\Omega|$. Indeed, this equation follows from
\begin{itemize}
\item
$\ip{\p_\tau\xi}{\omega^j}=-\ip{\xi}{\p_\tau \omega^j}=0,$ $j=0, \dots, n+1$;
\item $\ip{L_a\xi}{\omega^j}=\ip{\xi}{L_a \omega^j}=0,$ $j=0, \dots, n+1$;
\item
$\ip{{F}}{1}
=\ip{\frac{\sqrt{n}}{2}a^{-3/2}a_\tau}{1}=\frac{\sqrt{n}}{2}a^{-3/2}a_\tau|\Omega|;$
\item
$\ip{{F}}{\omega^j}
=-\lambda^{-1}\ip{z_\tau\cdot\omega}{\omega^j}=-c \lambda^{-1}z^j_\tau,\ j=1, \dots, n+1.$
\end{itemize}

\DETAILS{
\textbf{An alternative to proceeding as in the NLH (see the remark at the beginning of Section \ref{sec:reparametrization}) would be to consider the system \eqref{eqn:xi}-\eqref{projectionomegaj} and show the local existence for this system. The equations \eqref{eqn:xi}-\eqref{projectionomegaj} imply $\xi\bot \omega^j,\ j=0, \dots, n+1$ . Another alternative pointed out in Section \ref{sec:mainproof}.}
}

Equations \eqref{projection1} and \eqref{projectionomegaj} give
\begin{align}\label{estalpha}
\big|a^{-3/2}a_\tau\big| &\lesssim
\big|\ip{{N}(\xi)+\lambda^{-1}\tilde{z}_\tau\cdot\frac{\nabla \xi}{\rho_a+\xi}}{1}\big|\notag \\
&\lesssim\|{N}(\xi)\|_{L^1}+\lambda^{-1}|z_\tau|\|\nabla \xi\|_{L^1}
\end{align}
and
\begin{align}\label{estz}\lambda^{-1}\big|z_\tau\big| &\lesssim
\big|\ip{{N}(\xi)+\lambda^{-1}\tilde{z}_\tau\cdot\frac{\nabla \xi}{\rho_a+\xi}}{\omega}\big|\notag \\
&\lesssim \|{N}(\xi)\|_{L^1}+\lambda^{-1}|z_\tau|\|\nabla \xi\|_{L^1}.
\end{align}

Next, we estimate ${N}(\xi)$. Using \eqref{LNF}, 
where, recall, $\rho=\rho_a+\xi$, and assuming that $|\xi| \le \frac{1}{2}\rho_a$, we have that
\begin{equation}\label{est_N}\|{N}(\xi)\|_{L^1} \lesssim
(\|\nabla\xi\|^2_{L^{4}}+\|\xi\|_{H^1})\|\xi\|_{H^2}.
\end{equation}
This together with \eqref{estalpha} and \eqref{estz} gives
\begin{equation}\label{a_tau}
|a^{-3/2}a_\tau|\lesssim(\|\nabla\xi\|^2_{L^{4}}+\|\xi\|_{H^1})\|\xi\|_{H^2}
\end{equation}
and, provided that $\|\xi\|_{H^1} \ll 1$,
\begin{equation}\label{z_tau}
|z_\tau|\lesssim \lambda(\|\nabla\xi\|^2_{L^{4}}+\|\xi\|_{H^1})\|\xi\|_{H^2}.
\end{equation}

\section{Linearized operator}\label{sec:linearopr}

The linearization of the map $-G(\rho)$ at $\rho_a=\sqrt{\frac{n}{a}}$ is the  operator
$L_a:=-\p J(\sqrt{\frac{n}{a}})=\frac{a}{n}(-\Delta-2n)$.
The spectrum of $-\Delta$ is well known (see \cite{SW}):
$\{l(l+n-1)|\ l=0,1,\cdots\}.$ Let $H_l$ be the space of all
the eigenfunctions corresponding to the eigenvalue $l(l+n-1)$ of
$-\Delta$. Then $\dim \ H_l=\left(\begin{array}{c}  n+l \\
                                                  n
                                                \end{array}\right)-
                                                \left(\begin{array}{c}
                                                n+l-2 \\
                                                   n
                                                   \end{array}\right).$
Moreover, $H_0=\Span\{1\}$ and $H_1=\Span\{\frac{x^1}{|x|}, \cdots,
\frac{x^{n+1}}{|x|}\}$. Hence the spectrum of $L_a$ is
$\{\frac{a}{n}(l(l+n-1)-2n)|\ l=0,1,\cdots\}$ and
\begin{equation}\label{EVs}
L_a   \omega^j=-2a\delta_{j0},\ j=0, \dots, n+1.
\end{equation}
The conclusions above imply that
\begin{equation}\label{coercivity}
\ip{\xi}{L_a \xi} \geq \frac{2a}{n}\|\xi\|^2\
\text{if}\ \xi \perp \omega^j,\ j=0, \dots, n+1.
\end{equation}
Now, \eqref{coercivity} is the reason why we need the conditions \eqref{orthv}.

Observe that the eigenfuctions $\omega^j$ are related to the zero modes of the operator $L_\alpha:=\p J(\rho_\alpha)$, where $\alpha=(R, z)$,   and $\rho_\alpha$ is the map from
$\Omega$ to $\R$ satisfying
$|\rho_\alpha(x)\hat{x}-z|=R$. Note that, if $S_{R,z}$ denotes the sphere in $\R^{n+1}$ of radius $R$, centered at $z\in\R^{n+1}$ and
$graph(\rho):= \{ \rho(\omega)\omega: \omega \in S^n\}$ for $\rho: S^n \rightarrow \R^{+}$, then  $S_\alpha=graph(\rho_\alpha)$. Hence $ J(\rho_\alpha)=0$ for any $\alpha$.
Indeed, differentiating $ J(\rho_\alpha)=0$ we find $\p J(\rho_\alpha)\p_\alpha \rho_\alpha+\p_\alpha J(\rho_\alpha)=0$, which implies
 $L_{\alpha}\p_R \rho_\alpha=-2a\p_R \rho_\alpha$, $\ L_{\alpha}\p_z \rho_\alpha=0$, i. e. $\p_\alpha \rho_\alpha$ are eigenfunctions of the operator
$L_\alpha$. 

On the other hand, the equation for $\rho_\alpha$ implies
implies that
$\rho_\alpha(x)^2+|z|^2-2\rho_\alpha(x)z\cdot\hat{x}=R^2$ and therefore 
$$\rho_\alpha (\hat{x})=z\cdot\hat{x}+\sqrt{R^2-|z|^2+(z\cdot\hat{x})^2},$$ 
where, recall, $\hat{x}=\frac{x}{|x|}$. Differentiating the former relation with respect to $R$ and $z^j$, we obtain
\begin{equation}\label{prhoalpha}
\p_R\rho_\alpha(x)=\frac{R}{\rho_\alpha(x)-z\cdot\hat{x}}\ \textrm{and}\ \p_{z^j}\rho_\alpha(x)=\frac{\rho_\alpha(x)\hat{x}^j-z^j}
{\rho_\alpha(x)-z\cdot\hat{x}}.
\end{equation}
Hence we have that
\begin{equation}\label{zeromodes}
\p_R\rho_\alpha(x)=1+O(|z|),\ \p_{z^j}\rho_\alpha(x)=\hat{x}^j+O(|z|).
\end{equation}
Since $L_\alpha=L_a+O(|z|)$, these equations and $L_{\alpha}\p_R \rho_\alpha=-2a\p_R \rho_\alpha$, $\ L_{\alpha}\p_z \rho_\alpha=0$ imply \eqref{EVs}. This relates the zero modes \eqref{prhoalpha} to the eigenfunctions in \eqref{EVs}. Finally, we note that $\p_\alpha\rho_\alpha$ are tangent vectors of the manifold of spheres, $\{S_{R, z}\ |R\in \R^+,\ z\in \R^{n+1}\}$.

%
%
\DETAILS{
When $n=1$, if $a$ is a positive constant, then $v=\frac{1}{\sqrt{a}}$ is a static solution to \eqref{eqn_v}. The linearized operator
of $-J(v)$ at $\frac{1}{\sqrt{a}}$ is $L_a:=-\p J(\frac{1}{\sqrt{a}})=a(-\p_\theta^2-2)$. The spectrum of $L_a$ is
$\{a(n^2-2), n=0,1,\cdots\}.$ The corresponding eigenfunctions of $L_a$ for $-2a$ and $-a$ are $1$,
$\cos\theta$ and $\sin\theta$, respectively. Hence if $\xi \perp 1, \cos\theta, \sin\theta$, then $\ip{L_a\xi}{\xi}\geq 2a\|\xi\|^2.$

When $n=2$, if $a$ is a positive constant, then $v=\sqrt{\frac{2}{a}}$ is a static solution to \eqref{eqn_v}. The linearized operator
of $-J(v)$ at $\sqrt{\frac{2}{a}}$ is $L_a:=-\p J(\sqrt{\frac{2}{a}})=-\frac{a}{2}(-\Delta_{S^2}-4)$. The spectrum of $-\Delta_{S^2}$
is $\{l(l+1): l=0, 1, \cdots\}$ and the eigenfunctions corresponding to $l(l+1)$ are the spherical
harmonics $Y_l^m$, where $m=-l, \cdots, l$. So if $\xi \perp Y_0^0, Y_1^{-1}, Y_1^0, Y_1^1$, then $\ip{\xi}{L_a\xi}\geq a\|\xi\|^2.$
Note that $Y_0^0=\frac{1}{2}\sqrt{\frac{1}{\pi}}$, $Y_1^{-1}=\frac{1}{2}\sqrt{\frac{3}{2\pi}}\sin\theta \cos\psi$,
$Y_1^0=\frac{1}{2}\sqrt{\frac{3}{\pi}}\cos\theta$ and $Y_1^1=-\frac{1}{2}\sqrt{\frac{3}{2\pi}}\sin\theta \sin\psi$.
}

\section{Lyapunov functional}\label{sec:functional}
Let a function $\xi$ obey \eqref{eqn:xi} and \eqref{xiorthog}. Using that \eqref{xiorthog} implies $\ip{L_a\xi}{\xi}\geq \frac{2a}{n}\|\xi\|^2$, we derive in this section some differential inequalities for certain Sobolev norms of such a $\xi$. These inequalities allow us to prove a priori estimates for these Sobolev norms. For $k\geq 1$, we define the functional
$\Lambda_k(\xi)=\frac{1}{2}\ip{\xi}{L_a^k\xi}$.
\begin{prop}\label{Hnorm}
There exist constants $c>0$ and $C>0$ such that
$$ca^k\|\xi\|^2_{H^k}\leq \Lambda_k(\xi)\leq Ca^k\|\xi\|^2_{H^k}.$$
\end{prop}
\begin{proof} By a standard
computation, we see that there exists a $C>0$ such that $\ip{\xi}{L_a^k\xi}\leq
Ca^k\|\xi\|^2_{H^k}$. We prove the lower bound below. Recall that
$\ip{\xi}{L_a\xi}\geq \frac{2a}{n} \|\xi\|^2$. From the definition of $L_a$ we also have
$\ip{\xi}{L_a\xi}=C_1a\|\nabla\xi\|^2-C_2a\|\xi\|^2$ for some $C_1>0$ and $C_2>0$.
These two inequalities imply that
\begin{equation*}
\begin{array}{ll}
\ip{\xi}{L_a\xi} & =\mu\ip{\xi}{L_a\xi}+(1-\mu)\ip{\xi}{L_a\xi}\\
& \geq \mu C_1a\|\nabla\xi\|^2-\mu C_2
a\|\xi\|^2+(1-\mu)C a\|\xi\|^2\\
& =\mu C_1 a(\|\nabla\xi\|^2+\|\xi\|^2)
\end{array}
\end{equation*}
provided that $\mu=\frac{C}{C+C_1+C_2}$, where $C=\frac{2a}{n}$.

For the general case, observe that $L_a$ is a self-adjoint
operator and $L_a^k$ has the same eigenfunctions as $L_a$ with eigenvalues
$\{\frac{a^k}{n^k}(l(l+n-1)-n)^k: l=0, 1, \cdots\}$. Hence, by \eqref{coercivity}, $\ip{\xi}{L_a^k \xi}\geq (\frac{2a}{n})^k \|\xi\|^2.$
On the other hand, we have as before $\ip{\xi}{L_a^k \xi}\geq (\frac{a}{n})^k[\|\xi\|_{H^k}^2-C\|\xi\|^2].$ Then proceeding
as above we find $\ip{\xi}{L_a^k \xi}\gtrsim a^k\|\xi\|_{H^k}^2,$ which is the lower bound in the proposition.
\end{proof}

\begin{prop}\label{diffineq}
Let $k>\frac{n}{2}+1$.
Then there exists a constant $C>0$ such that
\begin{equation}\label{diffinequality}
\p_\tau \Lambda_k(\xi)\leq
-\frac{a}{n}\Lambda_k(\xi)-[\frac{1}{2}-
C(\Lambda_k(\xi)^{1/2}+\Lambda_k(\xi)^k)]\|L_a^{\frac{k+1}{2}}\xi\|^2.
\end{equation}
\end{prop}

\begin{proof}
We have
\begin{equation}\label{eqn:differinner}
\frac{1}{2}\p_\tau \ip{\xi}{L_a^k \xi}=\ip{\p_\tau \xi}{L_a^k \xi}+\frac{1}{2}\ip{\xi}{(\p_\tau L_a^k)\xi}.
\end{equation}
First, from \eqref{eqn:xi}
\begin{equation}
\ip{\p_\tau\xi}{L_a^k \xi}=-\ip{L_a\xi}{L_a^k \xi}+\ip{N (\xi)}{L_a^k \xi}+
\ip{\lambda^{-1}\tilde{z}_\tau\cdot\frac{\nabla \xi}{\rho_a+\xi}}{L_a^k\xi}+\ip{F}{L_a^k \xi}.
\end{equation}
We consider each term on the right hand side. We have by \eqref{coercivity}
\begin{equation}\label{ineqn:innerprdct}
\begin{array}{ll}
\ip{L_a\xi}{L_a^k \xi}&=\frac{1}{2}\|L_a^{\frac{k+1}{2}}\xi\|^2+\frac{1}{2}\ip{L_a^{\frac{k}{2}}\xi}{L_a
L_a^{\frac{k}{2}}\xi}\\
 &\geq \frac{1}{2}\|L_a^{\frac{k+1}{2}}\xi\|^2+\frac{a}{n}\ip{L_a^{\frac{k}{2}}\xi}{L_a^{\frac{k}{2}}\xi}.
\end{array}
\end{equation}

To estimate the next term we need the following inequality proven in Appendix B:
\begin{equation}\label{N-pro}
\|L_a^{\frac{k-1}{2}}N(\xi)\|\lesssim (\Lambda_k^{1/2}(\xi)+\Lambda_k^k(\xi))\|L_a^{\frac{k+1}{2}}\xi\|.
\end{equation}
This estimate implies that
\begin{equation}\label{ineqn:mathN}
\begin{array}{ll}
|\ip{N(\xi)}{L_a^k \xi}| & = |\ip{L_a^{\frac{k-1}{2}}N(\xi)}{L_a^{\frac{k+1}{2}}\xi}|\\
& \leq \|L_a^{\frac{k-1}{2}}N(\xi)\|\|L_a^{\frac{k+1}{2}}\xi\|\\
& \leq C(\Lambda_k^{1/2}(\xi)+\Lambda_k^k(\xi))\|L_a^{\frac{k+1}{2}}\xi\|^2.
\end{array}
\end{equation}
From \eqref{z_tau} and Proposition \ref{Hnorm} we obtain that
\begin{equation}\label{ineqn:ztau}
\ip{\lambda^{-1}\tilde{z}_\tau\cdot\frac{\nabla \xi}{\rho_a+\xi}}{L_a^k\xi}=\ip{L_a^{\frac{k-1}{2}}(\lambda^{-1}\tilde{z}_\tau\cdot\frac{\nabla \xi}{\rho_a+\xi})}{L_a^{\frac{k+1}{2}}\xi}
\leq C(\Lambda_k^{1/2}(\xi)+\Lambda_k^k(\xi))\|L_a^{\frac{k+1}{2}}\xi\|^2.
\end{equation}
We have, by \eqref{EVs}, \eqref{orthv} (i.e. $L_a \omega^j= -2a \delta_{j0} $,   $\ip{\omega^j}{\xi}=0$) and the self-adjointness of $L_a$, that 
$\ip{\omega^j}{ L_a^k \xi}=0,\ j=0, \dots, n+1$, and therefore
\begin{equation}\label{eqn:mathF}
\ip{N}{L_a^k\xi}=0.
\end{equation}
Finally, 
we have using \eqref{a_tau}
\begin{equation}\label{ineqn:innerprdct1}
\ip{\xi}{(\p_\tau L_a^k)\xi}=\frac{ka_\tau}{a}\ip{\xi}{L_a^k\xi}\leq C(\|\xi\|_{H^k}+\|\xi\|^{2k}_{H^k})\|L_a^{\frac{k+1}{2}}\xi\|^2.
\end{equation}
Relations \eqref{eqn:differinner}-\eqref{ineqn:innerprdct1} yield \eqref{diffinequality}.
\end{proof}

\section{Proof of  Theorem \ref{thmrho}}\label{sec:mainproof}
We begin with reparametrizing the initial condition.
Applying Proposition \ref{prop:centerz}, to the immersion $x_0 (\omega)$
and the number $\lambda_0=1$, we find $z_0\in \R^{n+1}$, s.t. $\int_\Omega \rho_0 (\omega)\omega^j=0,\ j=1, \dots, n+1,$
where $\rho_0 (\omega)=(x_0 (\omega)-z_0)\cdot \omega$. Then we use Proposition \ref{prop:decomposition}
for $\rho_0 (\omega)$ to obtain $a_0$ and $\xi_0 (\omega)$,  s.t. $\rho_0=\rho_{a_0}+\xi_0$, with $\xi_0 \perp 1$.
Here, recall, $\rho_a=\sqrt{\frac{n}{a}}$. The last two statements imply that $\xi_0 \perp \omega^j,\ j=0, \dots, n+1,$ where,
recall, $\omega^0 =1$. If the initial condition, $x_0 (\omega)$, is sufficiently close to the identity, then 
$a_0$ and $\xi_0 (\omega)$ satisfy $\Lambda_k(\xi_0)^{\frac{1}{2}}+\Lambda_k(\xi_0)^k\leq \frac{1}{10C}$,
$\Lambda_k(\xi_0)\ll 1$ and $|a_0-n|\leq \frac{1}{10}$ (see Proposition \ref{prop:decomposition}),  where the constant $C$ is the same as in Proposition \ref{diffineq}.

Now we use the local existence result for the mean curvature flow.
For $\delta>0$ sufficiently small, the solution, $x(\omega, t)$,
in the interval $[0, \delta]$, stays sufficiently close to the standard sphere $\Omega$.
Hence we can apply  Proposition \ref{prop:centerz}, with $\bar z(t)=0$, to this solution in order to find $z(t)$,
s.t.
$$\int_\Omega ((x (\omega, t)-z(t))\cdot \omega)\omega^j=0,\ j=1, \dots, n+1,\ \text{and}\ z(0)=z_0.$$
By Proposition \ref{prop:eqnrho}, $y (\omega, \tau):=\lambda(t)^{-1}(x (\omega, t)-z(t))=\rho(\omega, \tau) \omega$,
with  $\rho(\omega, \tau)=(x (\omega, t)-z(t))\cdot \omega$ and $\lambda(t)$ satisfying \eqref{eqnrho}.  
Finally we apply Proposition \ref{prop:timedecom} to $R(\omega, t):=(x(\omega, t)-z(t))\cdot\omega=\lambda(t)\rho (\omega, \tau)$
to obtain $a(\tau) $ and $\xi(\omega, \tau)$  s.t. $\rho(\omega, \tau)=\rho_{a(\tau)}+\xi(\omega, \tau)$, with $\xi \perp \omega^j,\ j=0, \dots, n+1$. We repeat this procedure on the interval $[\delta, \delta+\delta']$ with $\bar z(t):=z(\delta)$ and so forth.
This gives $T_1>0$, $z(t(\tau))$, $\rho(\omega,\tau)$, $a(\tau)$ and $\xi(\omega,\tau)$, $\tau\leq T_1$, s.t.
$x(\omega,t)=z(t)+\lambda(t)\rho(\omega,\tau(t))$ and $\rho(\omega,\tau)=\rho_{a(\tau)}+\xi(\omega,\tau)$, with $\rho$ and $\lambda$
satisfying \eqref{eqnrho} and $\xi\perp \omega^j$, $j=0, \cdots, n+1$.

Now, let
$$T=\sup\{\tau>0:
\Lambda_k(\xi(\tau))^{\frac{1}{2}}+\Lambda_k(\xi(\tau))^k\leq \frac{1}{5C}, \ |a(\tau)-n|\leq\frac{1}{2}\}.$$
By continuity, $T>0$. Assume $T<\infty$. Then $\forall \tau\leq T$ we have by Proposition \ref{diffineq} that $\p_\tau
\Lambda_k(\xi)\leq -\frac{a}{n}\Lambda_k(\xi)$.
We integrate this equation to obtain
$\Lambda_k(\xi)\leq \Lambda_k(\xi_0)e^{-\int_0^\tau \frac{a(s)}{n}ds}\leq
\Lambda_k(\xi_0)e^{-(1-\frac{1}{2n})\tau} \leq \Lambda_k(\xi_0)$.
This implies
$$\Lambda_k(\xi(T))^{\frac{1}{2}}+\Lambda_k(\xi(T))^k\leq
\Lambda_k(\xi_0)^{\frac{1}{2}}+\Lambda_k(\xi_0)^k\leq \frac{1}{10C}.$$
Moreover, by \eqref{a_tau},

Hence $|a(T)-n|\leq\frac{1}{4}$, and therefore proceeding as above we see that there exists $\delta>0$ such that $\Lambda_k(\xi(t))^{\frac{1}{2}}+\Lambda_k(\xi(t))^k\leq \frac{1}{5C}$ and $|a(t)-n|\leq\frac{1}{2}$, for $t\leq T+\delta$, a contradiction! So $T=\infty$ and
$\Lambda_k(\xi)\leq \Lambda_k(\xi_0)e^{-\int_0^\tau \frac{a(s)}{n}ds}.$ By
Proposition \ref{Hnorm} we know that $\|\xi\|^2_{H^k}\lesssim
\Lambda_k(\xi_0)e^{-\int_0^\tau \frac{a(s)}{n}ds}$.
\begin{align}\label{est_a}
|a(\tau)^{-\frac{1}{2}}-a(0)^{-\frac{1}{2}}|&\leq \frac{1}{2}\int_0^\tau
|a(s)^{-\frac{3}{2}}a_\tau(s)|ds\lesssim \int_0^\tau \Lambda_k(\xi)ds \notag\\ &\leq \Lambda_k(\xi_0)\int_0^\tau
e^{-(1-\frac{1}{2n})s}ds\ll 1,
\end{align}
and by \eqref{z_tau}
\begin{align}\label{est_z}
|z(\tau)-z(0)|&\leq \int_0^\tau
|z_\tau(s)|ds\lesssim \int_0^\tau \lambda(s)\Lambda_k(\xi)ds  \notag\\ &\leq \Lambda_k(\xi_0)\int_0^\tau
e^{-(n+\frac{1}{2}-\frac{1}{2n})s}ds\ll 1.
\end{align}
Observe that $\lambda_0^2-\lambda(t)^2=2\int_0^t a(\tau(s))ds$. Let $t_*$ be the zero of the function
$\lambda_0^2-2\int_0^t a(\tau(s))ds$. Since $|a(\tau)-n| \le \frac{1}{2}$, we have $t_* < \infty$ and $\lambda(t)\rightarrow 0$ as $t\rightarrow t_*$. Similarly to
\eqref{est_a}, we know that
$$|a(\tau_2)^{-1/2}-a(\tau_1)^{-1/2}|\lesssim \int_{\tau_1}^{\tau_2} e^{-(1-\frac{1}{2n})s}ds \rightarrow 0,$$
as $\tau_1, \tau_2\rightarrow \infty$. Hence there exists $a_*>0$, such that $|a(\tau)-a_*|\lesssim e^{-(1-\frac{1}{2n})\tau}$.
Similar arguments show that there exists $z_* \in \R^{n+1}$ such that $|z(\tau)-z_*|\lesssim e^{-(n+\frac{1}{2}-\frac{1}{2n})\tau}$.
Then $\lambda^2=\lambda_0^2-2\int_0^t a(\tau(s))ds=2\int_t^{t_*}a(\tau(s))ds=2a_*(t_*-t)+o(t_*-t)$.
The latter relation implies that $\tau=\int_0^t \frac{ds}{\lambda(s)^2}=\frac{1}{2a_*}\int_0^t \frac{ds}{(t_*-s)(1+o(1))}$ and therefore
$e^{-(1-\frac{1}{2n})\tau}=O((t_*-t)^{\frac{1}{2a_*}(1-\frac{1}{2n})})$.
So $\lambda(t)=\sqrt{2a_*(t_*-t)}+O((t_*-t)^{\frac{1}{2}+\frac{1}{2a_*}(1-\frac{1}{2n})})$, $\rho(\omega,\tau(t))
=\sqrt{\frac{n}{a(\tau(t))}}+\xi(\omega,\tau(t))$,
and $\|\xi(\omega,\tau(t))\|_{H^k}\lesssim (t_*-t)^{\frac{1}{2n}}$.
The latter inequality with $k=s$, together with estimates on $a, z$ and $\lambda$ obtained above
and the relation $R(\omega,t)=\lambda(t)\rho(\omega,t)$, proves Theorem \ref{thmrho}.

\appendix
\section*{Appendix A: Proof of Proposition \ref{prop:eqnrho} } 

We rewrite \eqref{eqn:y} as
\begin{equation}\label{eqn:y'}
\frac{\p y}{\p \tau}\cdot\nu(y) =-H(y)+ay\cdot\nu(y)-\lambda^{-1}\frac{\p z}{\p \tau}\cdot\nu(y).
\end{equation}
Recall that in our representation $y=\rho(\omega, \tau)\omega,\ \omega \in S^n$. 
We extend $\rho$ to $\R^{n+1}\setminus\{0\}$ by $\tilde{\rho}(x,\tau)=\rho(\alpha(x), \tau)$, where
$\alpha: \R^{n+1}\rightarrow S^n,\
\alpha(x) := \hat{x}=\frac{x}{|x|}$ . Then $y=\tilde\rho(x, \tau)\hat x$ and we can write $\tilde M_\tau=\{x\in \R^{n+1}:
\varphi(x,\tau)=0\}$, where $\varphi(x,\tau)=|x|-\tilde{\rho}(x,\tau)$. Now,
$\nu(y) =\frac{\nabla_x\varphi}{|\nabla_x\varphi|}$ and $\tilde H:=\divv_x(\frac{\nabla_x\varphi}{|\nabla_x\varphi|})$, where $\nabla_x$ is the standard gradient in $x$.  Therefore
\begin{equation}\label{eqntilderho}
\frac{\p \tilde{\rho}}{\p \tau}=\tilde{G}(\tilde{\rho})+a\tilde \rho-\lambda^{-1}z_\tau\cdot(\hat{x}-\nabla_x\tilde{\rho})\ \quad \mbox{on}\ \quad \tilde M_\tau,
\end{equation}
where 
$\tilde{G}(\tilde{\rho})=-\tilde H\sqrt{1+|\nabla_x\tilde{\rho}|^2}$.

\DETAILS{$$XXXX $$
\textbf{Old version.}

We extend $\rho$ to
$\R^{n+1}\setminus\{0\}$ by
$\tilde{\rho}(x,t)=\rho(\alpha(x) ,t)$, where
$\alpha: \R^{n+1}\rightarrow \Omega,\
\alpha(x) := \hat{x}=\frac{x}{|x|}$ . Then we can write $S_t=\{y\in \R^{n+1}:
\varphi(y,t)=0\}$, where $\varphi(y,t)=|y-z(t)|-\tilde{\rho}(y-z(t),t)$. Then
\eqref{MCF} is equivalent to
\begin{equation}\label{eqnphi}
\p_t \varphi=\tilde H|\nabla_y\varphi|\ \mbox{on}\ S_t,
\end{equation}
where $\nabla_y$ is the standard gradient in $y$ and $\tilde H:=div_y(\frac{\nabla_y\varphi}{|\nabla_y\varphi|})$. Let $x=y-z(t)$. We compute
$\p_t\varphi=-\dot{z}(t)\cdot \hat{x}+\nabla_x\tilde{\rho}\cdot{\dot{z}}(t)
-\p_t\tilde{\rho}$,
$\nabla_y\varphi=\hat{x}-\nabla_x\tilde{\rho}$ and therefore
\begin{equation}\label{eqntilderho}
\p_t\tilde{\rho}=\tilde{G}(\tilde{\rho})-\dot{z}(t)\cdot\hat{x}+
\dot{z}(t)\cdot\nabla_x\tilde{\rho}\ \mbox{on}\ S_t,
\end{equation}
where
$\tilde{G}(\tilde{\rho})=-\tilde H\sqrt{1+|\nabla_x\tilde{\rho}|^2}$.

$$XXXXXX$$}

Let $\Hess_{x}$ be the operator-valued matrix with the entries $(\Hess_{x})_{ij}:=\p_{x_i}\p_{x_j}$. We compute $\tilde H$:
\begin{align}\label{eqnH}
\tilde H &=\divv(\frac{\hat{x}-\nabla_x\tilde{\rho}}{\sqrt{1+|\nabla_x\tilde{\rho}|^2}})
=\frac{\frac{n}{|x|}-\Delta_x\tilde{\rho}}{\sqrt{1+|\nabla_x\tilde{\rho}|^2}}
\notag \\ &+\frac{-\frac{1}{2}\hat{x}\cdot\nabla_x|\nabla_x\tilde{\rho}|^2+\nabla_x\tilde{\rho}\cdot \Hess_x(\tilde{\rho})
\nabla_x\tilde{\rho}}{(1+|\nabla_x\tilde{\rho}|^2)^{3/2}}.
\end{align}
Since $\tilde{\rho}(\lambda x)=\tilde{\rho}(x)$, we have that
$x\cdot\nabla_x\tilde{\rho}=0$. Differentiating this equation with
respect to $x_i$ we find that
$x\cdot\nabla_x\p_{x_i}\tilde{\rho}=-\p_{x_i}\tilde{\rho}$, and
therefore
$x\cdot\nabla_x|\nabla_x\tilde{\rho}|^2=2|\nabla_x\tilde{\rho}|^2$.
Plugging this into \eqref{eqnH} gives \begin{equation}\label{eqnH1}
\tilde H=\frac{\frac{n}{|x|}-\Delta_x\tilde{\rho}}{\sqrt{1+|\nabla_x\tilde{\rho}|^2}}
+\frac{-\frac{1}{|x|}|\nabla_x\tilde{\rho}|^2+\nabla_x\tilde{\rho}\cdot \Hess_x(\tilde{\rho})
\nabla_x\tilde{\rho}}{(1+|\nabla_x\tilde{\rho}|^2)^{3/2}}.
\end{equation}

Let $r=|x|$. We note first that due to the well-known
representation (see \cite{chavel})
\begin{equation}\label{laplace}
\Delta_x=r^{-n}\p_r
r^n\p_r+\frac{1}{r^2}\Delta\ \mbox{on}\ \R^{n+1},
\end{equation}
we have that
$\Delta_x\tilde{\rho}\mid_{M_t}=\frac{1}{\tilde{\rho}^2}\Delta.$
Furthermore, since $|x|\frac{\p u^j}{\p x^i}$ is homogeneous of degree $0$, $\p_{x^i}\tilde{\rho}=\frac{\p u^j}{\p x^i}\frac{\p \rho}{\p u^j}=\frac{1}{|x|}\frac{\p u^j}{\p x^i}|_{S^n} g_{jk}\nabla^k\rho
=\frac{1}{|x|}\frac{\p u^j}{\p x^i}|_{S^n}\frac{\p x^m}{\p u^j}\frac{\p x^m}{\p u^k}\nabla^k \rho=\frac{1}{|x|}\frac{\p x^i}{\p u^k}\nabla^k\rho$ and
therefore $\nabla_x\tilde{\rho}\cdot{z}_\tau=\frac{1}{|x|}{z}^i_\tau\frac{\p x^i}{\p u^k}\nabla^k\rho$.
Next, we need the following lemma which is proved below. 
\begin{lemma}\label{lemma:computation}
\begin{equation}\label{nabla}
|\nabla_x\tilde{\rho}|^2=\frac{1}{|x|^2}|\nabla\rho|^2,
\end{equation}
\begin{equation}\label{hessian}
\nabla_x\tilde{\rho}\cdot
\Hess_x(\tilde{\rho})\nabla_x\tilde{\rho}
=\frac{1}{|x|^4}(\nabla\rho\cdot \Hess(\rho)\nabla
\rho).
\end{equation}
\end{lemma}


This lemma, together with equations $H=\tilde H|_{\tilde{M}_\tau}$, \eqref{eqnH1} and \eqref{laplace}, gives
$H(\rho) := \frac{-\rho}{\sqrt{\rho^2+|\nabla\rho|^2}}G(\rho)$
and therefore $\tilde G(\tilde \rho)|_{\tilde{M}_\tau} =G(\rho)$.
This, together with \eqref{eqntilderho}, gives \eqref{eqnrho} - \eqref{G}. 
Hence if $\tilde M_\tau$, defined by the immersion $y(\omega, \tau)=\rho(\omega, \tau)\omega$ of $S^n$,
satisfies \eqref{eqn:y} 
then $\rho(\tau)$ satisfies \eqref{eqnrho} - \eqref{G}.
Reversing the steps we see that if $\rho$ satisfies \eqref{eqnrho} - \eqref{G}, then 
the immersion $y(\omega, \tau)=\rho(\omega, \tau)\omega$ satisfies \eqref{eqn:y}.
$\Box$

\begin{proof}[Proof of Lemma \ref{lemma:computation}]
Recall the notation $\alpha(x)=\hat{x}=\frac{x}{|x|}$.
Let $\beta: U\rightarrow\R^{n+1}$ be a local parametrization of
$\Omega$, and we denote $\rho$ in the local coordinates, i.e. $\rho\circ\beta$, again as  $\rho: U\rightarrow\R$.
We write  $\tilde{\rho} :=\rho\circ\alpha= \rho \circ \beta \circ\beta^{-1} \circ \alpha$, which we rewrite as
$\tilde{\rho} = \rho \circ \sigma$, where $\sigma:=\beta^{-1} \circ \alpha: \R^{n+1}\rightarrow U$.
Now, writing $u=u(x)\equiv\sigma(x)$, 
we define $\frac{\p u^k}{\p x^i}\frac{\p u^l}{\p
x^i}=: \tilde{g}^{kl}$, where we use the convention of summing over repeated indices. We claim
\begin{equation}\label{eqn:g}
\tilde{g}^{ij}(x)g_{jk}(u)=\frac{1}{|x|^2}\delta_{ik}.
\end{equation}
Indeed, since $\beta(\sigma(x))=\alpha(x)$, we have
\begin{equation}\label{eqn:alpha}
(\frac{\p x^i}{\p u^m}\circ \sigma) \frac{\p u^m}{\p x^j}=
\frac{\p \alpha^i}{\p x^j}=\frac{1}{|x|}(\delta_{ij}-\frac{x^ix^j}{|x|^2}).
\end{equation}
Note that $\sigma$ is homogeneous of degree $0$, so $x\cdot\nabla_x\sigma=0$. This together with
\eqref{eqn:alpha} implies that
\begin{equation}\label{eqn:g1}
\begin{array}{ll}
&\tilde{g}^{ij}(x)g_{jk}(u) =\frac{\p u^i}{\p x^m}\frac{\p u^j}{\p
x^m}(\frac{\p x^n}{\p u^j}\circ\sigma)(\frac{\p x^n}{\p u^k}\circ\sigma)\\
=&\frac{1}{|x|}\frac{\p u^i}{\p x^m}(\delta_{mn}-\frac{x^mx^n}{|x|^2})\frac{\p x^n}{\p u^k}
=\frac{1}{|x|}\frac{\p u^i}{\p x^m}\frac{\p x^m}{\p u^k}.
\end{array}
\end{equation}
Since $|x|\frac{\p u^i}{\p x^m}$ is homogeneous of degree $0$, we have that $|x|\frac{\p u^i}{\p x^m}=
\frac{\p u^i}{\p x^m}|_{\Omega}$, and therefore $\frac{\p u^i}{\p x^m}=\frac{1}{|x|}(\frac{\p u^i}{\p x^m}|_{\Omega})$.
Using $\sigma\circ\beta=1_U$, we compute that $(\frac{\p \sigma^i}{\p x^j}\circ\beta)\frac{\p \beta^j}{\p u^k}=\delta_{ik}$,
which is equivalent to $\frac{\p u^i}{\p x^j}|_\Omega (\frac{\p x^j}{\p u^k}\circ\sigma)=\delta_{ik}$. This gives us
\begin{equation}\label{eqn:g2}
\frac{\p u^i}{\p x^m}\frac{\p x^m}{\p u^k}=\frac{1}{|x|}(\frac{\p u^i}{\p x^m}|_\Omega)\frac{\p x^m}{\p u^k}=\frac{1}{|x|}\delta_{ik}.
\end{equation}
\eqref{eqn:g1} and \eqref{eqn:g2} imply the equation \eqref{eqn:g}.

Using the relations $\frac{\p \rho}{\p u^i}=g_{ij}\nabla^j \rho$ (this follows from the definition of $\nabla \rho$),
$\frac{\p \tilde{\rho}}{\p x^i}=\frac{\p u^j}{\p x^i}\frac{\p\rho}{\p u^j}$, and \eqref{eqn:g} we compute
\begin{equation}
\begin{array}{ll}
&|\nabla_x\tilde{\rho}|^2=\frac{\p \tilde{\rho}}{\p x^i}\frac{\p \tilde{\rho}}{\p x^i}\\
=&\frac{\p u^k}{\p x^i}\frac{\p u^l}{\p
x^i}\frac{\p \rho}{\p u^k}\frac{\p \rho}{\p u^l}=\tilde{g}^{kl}(x)\frac{\p \rho}{\p u^k}\frac{\p \rho}{\p u^l}\\
=&\tilde{g}_{kl}(x)g_{km}(u)\nabla^m\rho g_{ln}(u)\nabla^n\rho=\frac{1}{|x|^2}\nabla^l\rho g_{ln} \nabla^n\rho\\
=&\frac{1}{|x|^2}|\nabla\rho|^2.
\end{array}
\end{equation}
This gives \eqref{nabla}.

Now we prove \eqref{hessian}. We have
\begin{equation}\label{eqn:AB}
\begin{array}{ll}
&\nabla_x\tilde{\rho}\cdot  \Hess_x(\tilde{\rho})\nabla_x\tilde{\rho}
=\frac{\p \tilde{\rho}}{\p x^i}\frac{\p^2 \tilde{\rho}}{\p x^i \p x^j}\frac{\p \tilde{\rho}}{\p x^j}\\
&=\frac{\p u^m}{\p x^i}\frac{\p\rho}{\p u^m}\frac{\p u^l}{\p
x^i}\frac{\p}{\p u^l}(\frac{\p u^k}{\p x^j}\frac{\p\rho}{\p
u^k})\frac{\p u^n}{\p x^j}\frac{\p\rho}{\p u^n}\\
&=\tilde{g}^{ml}\frac{\p\rho}{\p u^m}\frac{\p}{\p u^l}(\frac{\p u^k}{\p
x^j}\frac{\p\rho}{\p u^k})\frac{\p u^n}{\p x^j}\frac{\p\rho}{\p u^n}\\
&=\tilde{g}^{ml}\tilde{g}^{kn}\frac{\p\rho}{\p u^m}\frac{\p^2\rho}{\p
u^lu^k}\frac{\p\rho}{\p u^n}+\tilde{g}^{ml}\frac{\p\rho}{\p u^m}\frac{\p}{\p
u^l}(\frac{\p u^k}{\p x^j})\frac{\p\rho}{\p u^k}\frac{\p u^n}{\p
x^j}\frac{\p\rho}{\p u^n}\\
&=: A+B,
\end{array}
\end{equation}
Then
\begin{equation}\label{eqn:A}
A=\tilde{g}^{ml}\tilde{g}^{kn}g_{mp}\nabla^p\rho\frac{\p^2 \rho}{\p u^l \p u^k}g_{nq}\nabla^q \rho
=\frac{1}{|x|^4}\nabla^l\rho\frac{\p^2 \rho}{\p u^l \p u^k}\nabla^k \rho
\end{equation}
and
\begin{equation*}
\begin{array}{ll}
B&=\frac{1}{2}\tilde{g}^{ml}\frac{\p\rho}{\p u^m}\frac{\p\rho}{\p u^k}\frac{\p
\rho}{\p u^n}\frac{\p}{\p u^l}(\frac{\p u^k}{\p x^j})\frac{\p
u^n}{\p x^j}+\frac{1}{2}\tilde{g}^{ml}\frac{\p\rho}{\p u^m}\frac{\p\rho}{\p
u^n}\frac{\p \rho}{\p u^k}\frac{\p}{\p u^l}(\frac{\p u^n}{\p
x^j})\frac{\p u^k}{\p x^j}\\
&=\frac{1}{2}\tilde{g}^{ml}\frac{\p\rho}{\p u^m}\frac{\p\rho}{\p
u^k}\frac{\p \rho}{\p u^n}\frac{\p}{\p u^l}(\frac{\p u^k}{\p
x^j}\frac{\p u^n}{\p x^j})\\
&=\frac{1}{2}\tilde{g}^{ml}\frac{\p\rho}{\p u^m}\frac{\p\rho}{\p
u^k}\frac{\p \rho}{\p u^n}\frac{\p \tilde{g}^{kn}}{\p u^l}.
\end{array}
\end{equation*}
Now, by $\frac{\p \rho}{\p u^i}=g_{ij}\nabla^j\rho$, $B=B_1=B_2=B_3$, where
\begin{equation*}
\begin{array}{l}
B_1=\frac{1}{2}\tilde{g}^{ml}g_{mr}g_{ks}\frac{\p \tilde{g}^{kn}}{\p u^l}\frac{\p \rho}{\p u^n}(\nabla\rho)^r(\nabla\rho)^s
=\frac{1}{2|x|^2}g_{ks}\frac{\p \tilde{g}^{kn}}{\p u^r}\frac{\p \rho}{\p u^n}(\nabla\rho)^r(\nabla\rho)^s,\\
B_2=\frac{1}{2}\tilde{g}^{ml}g_{ms}g_{nr}\frac{\p \tilde{g}^{kn}}{\p u^l}\frac{\p \rho}{\p u^k}(\nabla\rho)^r(\nabla\rho)^s
=\frac{1}{2|x|^2}g_{nr}\frac{\p \tilde{g}^{kn}}{\p u^s}\frac{\p \rho}{\p u^k}(\nabla\rho)^r(\nabla\rho)^s,\\
B_3=\frac{1}{2}\tilde{g}^{ml}g_{kr}g_{ns}\frac{\p \tilde{g}^{kn}}{\p u^l}\frac{\p \rho}{\p u^m}(\nabla\rho)^r(\nabla\rho)^s.
\end{array}
\end{equation*}
Hence
\begin{equation}\label{eqn:B}
B=-\frac{1}{|x|^4}\Gamma_{rs}^p\frac{\p \rho}{\p u^p}\nabla^r\rho\nabla^s\rho,
\end{equation} where
$\Gamma_{rs}^p =-\frac{|x|^2}{2}
(g_{ks}\frac{\p \tilde{g}^{kp}}{\p u^r}+g_{nr}\frac{\p \tilde{g}^{pn}}{\p u^s}-|x|^2\tilde{g}^{pl}g_{kr}g_{ns}\frac{\p \tilde{g}^{kn}}{\p u^l}).$
Using that $$\frac{\p}{\p u^r}(g_{ks}\tilde{g}^{kp})=\frac{\p}{\p u^r}(\frac{1}{|x|^2}\delta_{sp})=0$$ (points $x\in\R^{n+1}$
are parameterized by $\beta(u)$ and $|x|$), we compute
$g_{ks}\frac{\p \tilde{g}^{kp}}{\p u^r}=\frac{\p}{\p u^r}(g_{ks}\tilde{g}^{kp})-\tilde{g}^{kp}\frac{\p g_{ks}}{\p u^r}
=-\tilde{g}^{kp}\frac{\p g_{ks}}{\p u^r}$. This gives
\begin{equation*}
\begin{array}{ll}
\Gamma_{rs}^p &=  \frac{|x|^2}{2}(\tilde{g}^{kp}\frac{\p g_{ks}}{\p u^r}+\tilde{g}^{kp}\frac{\p g_{kr}}{\p u^s}-
|x|^2\tilde{g}^{pl}g_{kr}\tilde{g}^{kn}\frac{\p g_{ns}}{\p u^l})\\
&= \frac{|x|^2}{2}(\tilde{g}^{kp}\frac{\p g_{ks}}{\p u^r}+\tilde{g}^{kp}\frac{\p g_{kr}}{\p u^s}-
\tilde{g}^{pk}\frac{\p g_{rs}}{\p u^k}).\\
\end{array}
\end{equation*}
Since $|x|=1$, and therefore $\tilde{g}^{pk}=g^{kp}$ on $\Gamma$, we have that on $\Gamma$
$$\Gamma_{rs}^p=\frac{1}{2}g^{kp}(\frac{\p g_{ks}}{\p u^r}+\frac{\p g_{kr}}{\p u^s}-
\frac{\p g_{rs}}{\p u^k}),$$
which coincides with our definition for $\Gamma_{rs}^p$ at the beginning of Section \ref{sec:eqnrho}.

Equations \eqref{eqn:AB}, \eqref{eqn:A} and \eqref{eqn:B} give \eqref{hessian}.
This finishes the proof of the lemma.
\end{proof}


\section*{Appendix B: Proof of  \eqref{N-pro}}
\begin{lemma}\label{lemma:N}Let $k>\frac{n}{2}+1$ and assume that $|\xi| \le \frac{1}{2}v_a$. Then
\begin{equation}\label{N-pro'}
\|L_a^{\frac{k-1}{2}}N (\xi)\|\lesssim (\Lambda_k^{1/2}(\xi)+\Lambda_k^k(\xi))\|L_a^{\frac{k+1}{2}}\xi\|.
\end{equation}
\end{lemma}

\begin{proof}
Assume first that $k$ is an integer. Then $\|L_a^{\frac{k-1}{2}}\eta\| \simeq \|\eta\|^2_{H^{k-1}} \simeq \|\eta\|_{L^{2}} + \|\nabla^{k-1}\eta\|^2_{L^{2}} $. Now, by the expression for $\N$ in \eqref{LNF}, which we recall here,
\begin{equation}\label{N} N(\xi)=-\frac{(\rho_a+\rho)\xi\Delta\xi}{\rho^2\rho_a^2}-\frac{n\xi^2}{\rho\rho_a^2}+\frac{|\nabla\xi|^2}{\rho(\rho^2+|\nabla\xi|^2)}
-\frac{\nabla\xi\cdot Hess(\xi)\nabla\xi}{\rho^2(\rho^2+|\nabla\xi|^2)},\end{equation}
the term $|\nabla^{k-1}N(\xi)|$ is bounded above by terms of the form
$|\xi^t(\nabla\xi)^r(\nabla^{\alpha_1}\xi)\cdots (\nabla^{\alpha_s}\xi)|$,
where
\begin{equation}\label{indices} 0\leq t,\ r\leq k+1,\ 1\leq s\leq k,\ t+r+s\geq 2,\ 2\leq \alpha_1\leq \cdots \leq \alpha_s \leq k-s+2,\
\alpha_1+\cdots+\alpha_s\leq k+s.
\end{equation}
Note that the last two conditions in \eqref{indices} imply that $s \le k$. Then by H\"older's inequality we have
$$\|\nabla^{k-1}N(\xi)\|\leq \|\nabla\xi\|^r_{L^\infty}\|\nabla^{\alpha_1}\xi\|_{L^{p_1}}
\cdots \|\nabla^{\alpha_s}\xi\|_{L^{p_s}},$$
where $\frac{1}{p_1}+\cdots+\frac{1}{p_s}=\frac{1}{2}.$

Since $k>\frac{n}{2}+1$, we have, by the Sobolev embedding theorem, that $\|\xi\|_{L^\infty}+\|\nabla\xi\|_{L^\infty}\lesssim \|\xi\|_{H^k}$.
Moreover, we choose $p_i$ so that $k-\alpha_i> \frac{n}{2}-\frac{n}{p_i}$ for all $i=1, \cdots, s-1$
and $k+1-\alpha_s>\frac{n}{2}-\frac{n}{p_s}$ (this choice implies $\sum_{j=1}^{s}\alpha_j <\frac{n}{2}+1 +(k- \frac{n}{2})s$, which is compatible with \eqref{indices} ). Then, using
the Sobolev embedding theorem again, we have $\|\nabla^{\alpha_i}\xi\|_{L^{p_i}}\leq \|\xi\|_{H^k}$,
for  $i=1, \cdots, s-1$, and $\|\nabla^{\alpha_s}\xi\|_{L^{p_s}}\leq \|\xi\|_{H^{k+1}}$.  Combining these estimates gives us
$$\|L_a^{\frac{k-1}{2}}N(\xi)\|\lesssim \|\xi\|_{H^k}^{r+s-1}\|\xi\|_{H^k}.$$
Now from $1\leq r+s-1\leq 2k$ and Proposition \ref{Hnorm} we obtain \eqref{N-pro'}. 
Furthermore, one can easily check that $k$ can be taken arbitrary close to $\frac{n}{2}+1$ (this means that one is able to satisfy $1 \ge \alpha_i-\frac{n}{p_i}$, for $i=1, \cdots, s-1$, $2 \ge \alpha_s-\frac{n}{p_s}$ and $  \alpha_i\ge 2,\ \forall i$).

\DETAILS{The conditions on $k$
and $\alpha_{i}$'s guarantee the existence of such $p_i$'s: $\sum_{j=1}^{s-1}(k-\alpha_j)+k+1-\alpha_s > \sum_{j=1}^{s}(\frac{n}{2}- \frac{n}{p_j})$ which implies $ks+1- \sum_{j=1}^{s}\alpha_j > \frac{n}{2} (s-1)$. Furthermore, by \eqref{indices} and $k>\frac{n}{2} +1$, we have $ks+1- \sum_{j=1}^{s}\alpha_j  \ge ks+1-k-s > \frac{n}{2} (s-1)+s-1+1-s > \frac{n}{2} (s-1),$ which agrees with the previous inequality.}

If $k$ is not integer, we proceed as follows. Let $\beta=k-[k]\in (0,1)$. We use the space $\tilde{H}^\beta$ with the norm
$$\|f\|_{\tilde{H}^\beta}=\|f\|_{L^2}+\int \frac{dh}{|h|^{n+\beta}}\|\Delta_h f\|_{L^2},$$
where $\Delta_h f(x) = f(x+h)-f(x)$. We have the embeddings
\begin{equation} \label{embedding}
\|f\|_{H^\beta} \lesssim \|f\|_{\tilde{H}^\beta} \lesssim \|f\|_{H^{\beta'}}, \ \beta<\beta'.
\end{equation}
Let us prove the first embedding:
$$(-\Delta+1)^{\beta/2} f(x)=C_\beta f(x) + \int (f(x-y)-f(x))G_{\beta}(y)dy,$$
where $C_\beta $ is an analytic continuation of $C_\beta:=\int G_\beta(x)dx$ with $\re(\beta)<n$ and $G_{\beta}(y):=\int e^{iy\cdot k}(|k|^2+1)^{\beta/2} dk$. Note that $G_{\beta}(y) \sim |y|^{-n-\beta}$ as $|y|\rightarrow 0$ and is exponentially
decaying at $\infty$. So
$$\|f\|_{H^\beta}=\|(-\Delta+1)^{\beta/2} f\|_{L^2} \leq C_\beta \|f\|_{L^2} + \int \frac{dy}{|y|^{n+\beta}}\|\Delta _y f\|_{L^2}
\lesssim \|f\|_{\tilde{H}^\beta},$$
which proves the first embedding in \eqref{embedding}.

For the second embedding, let $\varphi=(-\Delta+1)^{\beta'/2} f$. Then
$$f=(-\Delta+1)^{-\beta'/2}\varphi=\int \tilde{G}_{\beta'}(x-y)\varphi(y)dy,$$
where $\tilde{G}_{\beta'}(y):=\int e^{iy\cdot k}(|k|^2+1)^{-\beta'/2} dk$. Note that $\tilde{G}_{\beta'}(y) \sim |y|^{-n+\beta'}$ as $|y|\rightarrow 0$ and is exponentially decaying at $\infty$. Let $\beta < \beta'' < \beta'.$
Then
\begin{equation}\label{2ndembedding}
\begin{array}{ll}
& \int_{|h|\leq 1} \frac{dh}{|h|^{n+\beta}}\|\Delta_h f\|_{L^2}\\
= & \int_{|h|\leq 1} \frac{dh}{|h|^{n+\beta}}\|\int_{|x-y| \leq 2} (\tilde{G}_{\beta'}(x+h-y)-\tilde{G}_{\beta'}(x-y))\varphi(y) dy\\
&+\int_{|x-y| \geq 2} (\tilde{G}_{\beta'}(x+h-y)-\tilde{G}_{\beta'}(x-y))\varphi(y) dy\|_{L^2}\\
\lesssim  & \int_{|h|\leq 1} \frac{dh}{|h|^{n+\beta}}(|h|^{\beta''}\|\int_{|x-y|\leq 2} |x-y|^{-n+\beta'-\beta''}|\varphi(y)|dy\|_{L^2}
+|h|\|\int_{|x-y|\geq 2} |x-y|^{-n+\beta'-1}|\varphi(y)|dy\|_{L^2})\\
\lesssim & \|\varphi\|_{L^2}=\|f\|_{H^{\beta'}}
\end{array}
\end{equation}
and
$$\int_{|h|\geq 1}\frac{dh}{|h|^{n+\beta}}\|\Delta_h f\|_{L^2} \leq 2\|f\|_{L^2}\int_{|h|\geq 1}\frac{dh}{|h^{n+\beta}|}\lesssim \|f\|_{H^{\beta'}}.$$
This proves the second embedding in \eqref{embedding}.

Using \eqref{embedding}, we obtain
\begin{equation*}
\begin{array}{ll}
\|\prod_{j=1}^s \xi_j\|_{H^\beta} & \lesssim \int \frac{dh}{|h|^{n+\beta}}\|\Delta_h \prod_{j=1}^s \xi_j\|_2\\
& \leq \sum_{i=1}^s \int \frac{dh}{|h|^{n+\beta}}\|\prod_{j=1}^{i-1}\xi_j \Delta_h \xi_i \prod_{j=i+1}^s T_h \xi_j\|_2\\
& \leq \sum_{i=1}^s (\prod_{j \neq i}\|\xi_j\|_{p_{j}^{(i)}}) \int \frac{dh}{|h|^{n+\beta}} \|\Delta_h \xi_i\|_{p_{i}^{(i)}}, 
\end{array}
\end{equation*}
where $T_h f(x)=f(x+h)$, $\sum_{j=1}^s \frac{1}{p_{j}^{(i)}}=\frac{1}{2}$. Using appropriate embeddings, we conclude finally that
\begin{equation} \label{prodHbeta}
\|\prod_{j=1}^s \xi_j\|_{H^\beta}  \lesssim \sum_{i=1}^s \prod_{j =1}^{s}\|\xi_j\|_{H^{c^{(i)}_j}}, 
\end{equation}
where  $c_j^{(i)}> \frac{n}{2}-\frac{n}{p_j^{(i)}} \ \forall j\neq i$
and $c_i^{(i)}-\beta > \frac{n}{2}-\frac{n}{p_i^{(i)}}$. Similarly as before we know that $\sum_{j=1}^s c_j^{(i)}-\beta > \frac{n}{2}(s-1)$, which guarantees the
existence of $p_j^{(i)}$.

For $k$  not an integer, we write
\begin{equation} \label{NHk-1}
\|N(\xi)\|_{H^{k-1}}\sim \|(-\Delta+1)^{\beta/2}\nabla^m N(\xi)\|_{L^2},
\end{equation}
where $m=[k]-1$ and $\beta=k-[k]\in (0,1)$. $\nabla^m N(\xi)$ is treated as before to obtain
\begin{equation} \label{nablamN}
\nabla^m N(\xi) \sim \xi^t (\nabla\xi)^r \nabla^{\alpha_1}\xi \cdots \nabla^{\alpha_s}\xi, 
\end{equation}
where $t\leq m+2$, $r\leq m+2$, $2 \leq \alpha_j \leq m-s+3$, $\sum_{j=1}^s \alpha_j \leq m+1+s$, $s\leq m+1$ and $t+r+s \geq 2$.

If $\alpha_j < m+2 \ \forall j$,
then,  using \eqref{prodHbeta} with $\xi_j=\nabla^{\alpha_j}\xi \ \forall j$, $c_j^{(i)}+\alpha_j=k\ \forall j\neq i$
and $c_i^{(i)}+\alpha_i = k+1$, we find
\begin{equation} \label{Hbeta}
\|\xi^t (\nabla \xi)^r\prod_{j=1}^s \nabla^{\alpha_j}\xi\|_{H^\beta}  \lesssim \|\xi\|_{H^{k}}^{r+s-1} \|\xi\|_{H^{k+1}}.
\end{equation}
We use this estimate, together with \eqref{NHk-1} and \eqref{nablamN}, to obtain
\begin{equation} \label{Nest} 
\|N(\xi)\|_{H^{k-1}} \lesssim \sum_{i=1}^{2[k]}\|\xi\|_{H^k}^i\|\xi\|_{H^{k+1}}.
\end{equation}

If $\alpha_s=m+2$ and therefore $s=1$, then we let $f=\xi^t(\nabla\xi)^r$ and proceed as
\begin{equation} \label{2ndcase}
(-\Delta+1)^{\beta/2}f\nabla^{m+2}\xi=f(-\Delta+1)^{\beta/2}\nabla^{m+2}\xi+[(-\Delta+1)^{\beta/2}, f]\nabla^{m+2}\xi.
\end{equation}
The first term on the r.h.s. is easy to estimate:
\begin{equation}  \label{est1stterm}
\begin{array}{ll}
&\|f(-\Delta+1)^{\beta/2}\nabla^{m+2}\xi\|\leq \|f\|_{\infty}\|\xi\|_{H^{k+1}}\\
\leq & \|\xi\|_{H^k}^{t+r}\|\xi\|_{H^{k+1}} \leq \sum_{p=1}^2 \|\xi\|_{H^k}^p \|\xi\|_{H^{k+1}}.
\end{array}
\end{equation}
To estimate the second term in the r.h.s. we note that
\begin{equation*}
\begin{array}{ll}
& [(-\Delta+1)^{\beta/2}, f]\eta = \int (f(x)-f(y))G_\beta(x-y)\eta(y)dy\\
= & \int \eta(x-z)(f(x-z)-f(x))G_\beta(z)dz.
\end{array}
\end{equation*}
Using this representation we obtain for $\beta'>\beta$,
\begin{equation*}
\begin{array}{ll}
& \|[(-\Delta+1)^{\beta/2}, f]\eta\|_2\\
\leq & \sup_z \|\eta(\cdot-z)\frac{f(\cdot-z)-f(\cdot)}{|z|^{\beta'}}\|_{L^2(dx)} \int |z|^{\beta'}|G_\beta(z)|dz\\
\lesssim & \sup_z \|\eta(\cdot-z)\frac{f(\cdot-z)-f(\cdot)}{|z|^{\beta'}}\|_{L^2(dx)} \\
\leq & \|\eta\|_q \sup_z \|\frac{\Delta_z f}{|z|^{\beta'}}\|_p,
\end{array}
\end{equation*}
where $\frac{1}{p}+\frac{1}{q}=\frac{1}{2}$.
Similar to \eqref{2ndembedding}, we have
$$\sup_z \|\frac{1}{|z|^\gamma}\Delta_z f\|_{H^b} \lesssim \|f\|_{H^{b+\gamma'}}, \ \gamma' > \gamma.$$
Using this estimate and Sobolev embedding theorem, we find
$$\|[(-\Delta+1)^{\beta/2}, f]\eta\|_2 \lesssim \|\eta\|_{H^a} \sup_z \|\frac{\Delta_z f}{|z|^{\beta'}}\|_{H^b}
\lesssim \|\eta\|_{H^a}\|f\|_{H^{b+\beta''}},$$
where $\beta''>\beta'$, $a>\frac{n}{2}-\frac{n}{q}$, $b>\frac{n}{2}-\frac{n}{p}$. Taking $f=\xi^t(\nabla\xi)^r$
and $\eta=\nabla^{m+2}\xi$, $a=\beta$, we find
$$\|[(-\Delta+1)^{\beta/2}, f]\nabla^{m+2}\xi\| \leq \|\xi^t(\nabla\xi)^r\|_{H^{r+\beta''}}\|\xi\|_{H^{k+1}}.$$
Note that $\beta''+r>n-\frac{n}{2}=\frac{n}{2}$. Let $\beta''+r=j$. As before, we estimate
\begin{equation*}
\|\xi^t(\nabla\xi)^r\|_{H^j} \lesssim \sum_{j_1+\cdots+j_{t+r}=j}\|\nabla^{j_1}\xi\cdots\nabla^{j_{t}}\xi
\nabla^{j_{t+1}+1}\xi\cdots\nabla^{j_{t+r}+1}\xi\|_{2}
\lesssim \|\xi\|_{H^{j+1}}^{t+r} \ \forall j>\frac{n}{2}.
\end{equation*}
Since $k>\frac{n}{2}+1$, we can take $j=k-1$ and so
$$\|[(-\Delta+1)^{\beta/2}, f]\nabla^{m+2}\xi\| \leq \|\xi\|_{H^k}^{t+r}\|\xi\|_{H^{k+1}},$$
where, recall, $f=\xi^t(\nabla\xi)^r$. This inequality together with \eqref{NHk-1},  \eqref{2ndcase} and \eqref{est1stterm} implies \eqref{Nest} also in this case. As was mentioned above \eqref{Nest} implies \eqref{N-pro'}. 
\end{proof}

\bibliographystyle{abbrv}
\bibliography{MCF}
\end{document}